\documentclass{amsart}

\usepackage[utf8]{inputenc}
\usepackage{amsmath}
\usepackage{amssymb}
\usepackage{amsthm}
\usepackage{hyperref}
\usepackage[arrow, matrix, curve]{xy}
\usepackage{tikz}
\usepackage{verbatim}
\usepackage{footnote}

\usepackage{amsopn}
\DeclareMathOperator{\Aut}{Aut}
\DeclareMathOperator{\Out}{Out}
\DeclareMathOperator{\Inn}{Inn}
\DeclareMathOperator{\Isom}{Isom}
\DeclareMathOperator{\BS}{BS}

\usepackage{setspace}
\allowdisplaybreaks[4]

\newtheorem{thm}{Theorem}
\newtheorem{prop}[thm]{Proposition}
\newtheorem{deprop}[thm]{Proposition/Definition}
\newtheorem{lem}[thm]{Lemma}
\newtheorem{cor}[thm]{Corollary}
\newtheorem*{thm-metric}{Theorems~\ref{thm: irreducible infimum realized} and \ref{thm: inf = sup}}
\newtheorem*{thm-geodesics}{Theorem~\ref{thm: geodesics}}
\newtheorem*{thm-traintracks}{Theorem~\ref{thm: train tracks}}

\theoremstyle{definition}
\newtheorem{de}[thm]{Definition}
\newtheorem{ex}[thm]{Example}
\newtheorem*{rem}{Remark}
\newtheorem*{ex-virtuallyfree}{Example~\ref{ex: train tracks for virtually free}}
\newtheorem*{ex-GBS}{Example~\ref{ex: train tracks for GBS}}

\begin{document}

\title{The Lipschitz metric on deformation spaces of $G$-trees}

\author{Sebastian Meinert}
\address{Freie Universit\"at Berlin, Institut f\"ur Mathematik, Arnimallee 7, 14195 Berlin, Germany}
\email{\href{mailto:sebastian.meinert@fu-berlin.de}{sebastian.meinert@fu-berlin.de}}
\urladdr{\href{http://userpage.fu-berlin.de/meinert}{http://userpage.fu-berlin.de/meinert}}
\keywords{Lipschitz metric, deformation spaces, $G$-trees, outer automorphisms, train tracks, virtually free groups, generalized Baumslag-Solitar groups}

\begin{abstract}
For a finitely generated group $G$, we introduce an asymmetric pseudometric on projectivized deformation spaces of $G$-trees, using stretching factors of $G$-equivariant Lipschitz maps, that generalizes the Lipschitz metric on Outer space and is an analogue of the Thurston metric on Teichm\"uller space. We show that in the case of irreducible $G$-trees distances are always realized by minimal stretch maps, can be computed in terms of hyperbolic translation lengths, and geodesics exist. We then study displacement functions on projectivized deformation spaces of $G$-trees and classify automorphisms of $G$. As an application, we prove existence of train track representatives for irreducible automorphisms of virtually free groups and nonelementary generalized Baumslag-Solitar groups that contain no solvable Baumslag-Solitar group $\BS(1,n)$ with $n\geq 2$.
\end{abstract}

\maketitle

\thispagestyle{empty}

\tableofcontents

\section{Introduction}

Let $G$ be a finitely generated group. A \emph{$G$-tree} is a metric simplicial tree on which $G$ acts by simplicial isometries without inversions of edges. A $G$-tree is \emph{minimal} if it does not contain a proper $G$-invariant subtree. To a nontrivial minimal $G$-tree $T$ we associate its \emph{deformation space} $\mathcal{D}$ consisting of $G$-equivariant isometry classes of nontrivial minimal $G$-trees $T'$ for which there exist $G$-equivariant (not necessarily simplicial) maps $T\to T'$ and $T'\to T$. The outer automorphism group $\Out(G)$ contains a subgroup $\Out_\mathcal{D}(G)$ (see Definition~\ref{de: Out D}) that acts on $\mathcal{D}$ by precomposing the $G$-actions on the trees. Deformation space of $G$-trees are analogues of Teichm\"uller spaces of surfaces in the context of group splittings. Important examples include the (unprojectivized) Culler-Vogtmann Outer space of a nonabelian free group (Example~\ref{ex: Outer space}), deformation spaces of -- more generally -- virtually nonabelian free groups (Example~\ref{ex: virtually free group}), and deformation spaces of nonelementary generalized Baumslag-Solitar groups (Example~\ref{ex: nonelementary GBS groups}).

The \emph{topology} of deformation spaces of $G$-trees has been extensively studied, e.g., in \cite{Cl05} and \cite{GL07}. For instance, the projectivized deformation space $\mathcal{PD}=\mathcal{D}/\mathbb{R}_{>0}$, the space of $G$-equivariant homothety classes of $G$-trees in $\mathcal{D}$, can be given the structure of a contractible simplicial complex with missing faces. The \emph{geometry} of deformation spaces, however, has only been addressed in the special case of Outer space \cite{FM11}\footnote{And recently also in the case of the Outer space of a free product \cite{FM13}.}, which admits a description as a space of finite marked metric graphs. Here one studies the asymmetric Lipschitz metric, an analogue of the asymmetric Thurston metric on Teichm\"uller space. The purpose of this paper is to introduce an asymmetric pseudometric on general projectivized deformation spaces of $G$-trees that generalizes the asymmetric Lipschitz metric on Outer space. In order to do so, we think of $G$-trees in $\mathcal{PD}$ as their covolume-1-representatives in $\mathcal{D}$ and for $T,T'\in\mathcal{PD}$ we define $$d_{Lip}(T,T')=\log\left(\inf_f \sigma(f)\right)$$ where $f$ ranges over all $G$-equivariant Lipschitz maps from $T$ to $T'$ and $\sigma(f)$ denotes the Lipschitz constant of $f$. Although in general we have $d_{Lip}(T,T')\neq d_{Lip}(T',T)$ and $d_{Lip}(T,T')=0$ does not imply that $T$ and $T'$ are $G$-equivariantly isometric (Example~\ref{ex: not an asymmetric metric}), the Lipschitz metric turns out to have useful properties.

If $\mathcal{PD}$ consists of $G$-trees that are \emph{irreducible}, i.e., if $G$ contains a free subgroup of rank 2 acting freely, then the symmetrized Lipschitz metric $$d_{Lip}^{sym}(T,T')=d_{Lip}(T,T')+d_{Lip}(T',T)$$ is an actual metric on $\mathcal{PD}$ (Proposition~\ref{prop: symmetrized metric}).

\subsubsection*{Minimal stretch maps and witnesses}

A key feature of the Lipschitz metric on Outer space is that the distance between two marked metric graphs is always realized by a map with minimal Lipschitz constant and that the minimum Lipschitz constant equals the maximum ratio of lengths of immersed loops in the corresponding quotient graphs \cite[Proposition~3.15]{FM11}. This reflects a theorem of Thurston that the Lipschitz distance between two hyperbolic surfaces in Teichm\"uller space is always realized by a minimal stretch map and that the extremal Lipschitz constant equals the supremum ratio of lengths of essential simple closed curves \cite[Theorem~8.5]{Th98}. In the same spirit, we will show the following:

\begin{thm-metric}
Let $\mathcal{PD}$ be a projectivized deformation space of irreducible $G$-trees. For all $T,T'\in\mathcal{PD}$ there exists
\begin{enumerate}
\item a $G$-equivariant Lipschitz map $f\colon T\to T'$ such that $d_{Lip}(T,T')=\log\left(\sigma(f)\right)$;
\item a hyperbolic group element $\xi\in G$ such that $$d_{Lip}(T,T')=\log\left(\frac{l_{T'}(\xi)}{l_T(\xi)}\right)=\log\left(\sup_g\frac{l_{T'}(g)}{l_T(g)}\right)$$ where $g$ ranges over all hyperbolic group elements of $G$ and by $l_T(g)$ we denote the \emph{translation length} $\inf_{x\in T}d(x,gx)$ of $g$ in $T$. We will call $\xi$ a \emph{witness} for the minimal stretching factor from $T$ to $T'$.
\end{enumerate}
\end{thm-metric}

\subsubsection*{Geodesics}

Francaviglia-Martino \cite{FM11} showed, making use of a folding technique due to Skora \cite{Sk89}, that the asymmetric Lipschitz metric on Outer space is geodesic. We will apply Skora's folding technique in the general context to show:

\begin{thm-geodesics}
If $\mathcal{PD}$ is a projectivized deformation space of irreducible $G$-trees then for all $T,T'\in\mathcal{PD}$ there exists a $d_{Lip}$-geodesic (see Definition~\ref{de: geodesic}) $\gamma\colon [0,1]\to\mathcal{PD}$ with $\gamma(0)=T$ and $\gamma(1)=T'$.
\end{thm-geodesics}

\subsubsection*{Train track representatives}

An automorphism $\Phi\in \Out_\mathcal{D}(G)$ is \emph{reducible} if there exists a $G$-tree $T\in\mathcal{PD}$ and a $G$-equivariant map $f\colon T\to T\Phi$ that leaves an essential proper $G$-invariant subforest of $T$ invariant, where a subforest $S\subset T$ is \emph{essential} if it contains the hyperbolic axis of some hyperbolic group element. We say that $\Phi\in \Out_\mathcal{D}(G)$ is \emph{represented by a train track map} if there exists a $G$-tree $T\in\mathcal{PD}$ and an extremal $G$-equivariant Lipschitz map $f\colon T\to T\Phi$ such that, loosely speaking, every iterate of $f$ maps certain immersed paths in $T$ to immersed paths (see Section~\ref{sec: train track representatives} for a precise definition). Bestvina \cite{Be11} classified free group automorphisms $\Phi$ by studying associated displacement functions $T\mapsto d_{Lip}(T,T\Phi)$ on Outer space. By doing so, he gave an alternative proof of Bestvina-Handel's train track theorem \cite[Theorem~1.7]{BH92} that every irreducible automorphism of a free group is represented by a train track map.
Generalizing Bestvina's approach, we will study displacement functions on projectivized deformation spaces of $G$-trees to classify automorphisms of more general groups and show the following:

\begin{thm-traintracks}
Let $\mathcal{PD}$ be a projectivized deformation space of irreducible $G$-trees. If $\Out_\mathcal{D}(G)$ acts on $\mathcal{PD}$ with finitely many orbits of simplices then every irreducible automorphism $\Phi\in \Out_\mathcal{D}(G)$ is represented by a train track map.
\end{thm-traintracks}

Throughout, we will pay particular attention to deformation spaces of $G$-trees for virtually free groups and nonelementary generalized Baumslag-Solitar groups:

\begin{ex-virtuallyfree}
Let $G$ be a finitely generated virtually nonabelian free group, i.e., $G$ contains a finitely generated nonabelian free subgroup of finite index. Let $\mathcal{PD}$ be the projectivized deformation space of minimal $G$-trees with finite vertex stabilizers (Example~\ref{ex: virtually free group}). Then $\Out_\mathcal{D}(G)=\Out(G)$ and every irreducible automorphism $\Phi\in \Out(G)$ is represented by a train track map. This generalizes \cite[Theorem~1.7]{BH92} to virtually free groups.
\end{ex-virtuallyfree}

A \emph{generalized Baumslag-Solitar (GBS) group} is a finitely generated group that acts on a simplicial tree with infinite cyclic vertex and edge stabilizers. Among these groups are the classical \emph{Baumslag-Solitar groups} $\BS(p,q)=\langle x,t\ |\ tx^pt^{-1}=x^q\rangle$ with $p,q\in\mathbb{Z}\setminus\left\{0\right\}$. A GBS group is \emph{nonelementary} if it is not isomorphic to $\mathbb{Z}$, $\BS(1,1)\cong\mathbb{Z}^2$, or the Klein bottle group $\BS(1,-1)$.

\begin{ex-GBS}
Let $G$ be a nonelementary GBS group that contains no solvable Baumslag-Solitar group $\BS(1,n)$ with $n\geq 2$. Let $\mathcal{PD}$ be the projectivized deformation space of minimal $G$-trees with infinite cyclic vertex and edge stabilizers (Example~\ref{ex: nonelementary GBS groups}). We have $\Out_\mathcal{D}(G)=\Out(G)$ and every irreducible automorphism $\Phi\in \Out(G)$ is represented by a train track map.
\end{ex-GBS}

\begin{rem}
In an earlier preprint of this paper, the main results were formulated under the additional hypothesis that the $G$-trees in $\mathcal{PD}$ are not only irreducible but also locally finite. Recently, Francaviglia-Martino \cite{FM13} independently proved analogous statements for irreducible $G$-trees with trivial edge stabilizers that are possibly locally infinite (they work in the Outer space of a free product, defined in \cite{GL07fp}). Making use of an ultralimit argument of Horbez \cite{Ho14}, we are now able to remove the local finiteness assumption for arbitrary edge stabilizers. As a special case, this gives an alternative, shorter proof of \cite[Theorem~5.12]{FM13}.
\end{rem}

\noindent\textbf{Structure of this paper}\quad In Sections~\ref{sec: trees} and \ref{sec: deformation spaces} we briefly review the notions of $G$-trees and deformation spaces of $G$-trees. In Section~\ref{sec: Lipschitz metric} we introduce the Lipschitz metric on projectivized deformation spaces of $G$-trees. We show that in the case of irreducible $G$-trees distances are always realized by minimal stretch maps, can be computed in terms of hyperbolic translation lengths, and geodesics exist. In Section~\ref{sec: classification} we study displacement functions on projectivized deformation spaces of $G$-trees, classify automorphisms of $G$, and address the existence of train track representatives for irreducible automorphisms.

\medskip

\noindent\textbf{Acknowledgements}\quad I am indebted to Mladen Bestvina for suggesting this line of research and for giving me valuable feedback. I would also like to thank Lee Mosher and Henry Wilton for their helpful comments on MathOverflow. The discussion on elliptic automorphisms in Section~\ref{sec: elliptics} arose out of discussions with Camille Horbez, whom I would like to thank for his deep interest, and Gilbert Levitt, to whom I am particularly thankful for his advice and for addressing my many questions. Finally, I would like to thank the referee for his comments.

\section{$G$-trees}
\label{sec: trees}

A \emph{metric simplicial tree} is a contractible 1-dimensional simplicial complex $T$ together with a positive length assigned to every edge. We denote by $V(T)$ the set of vertices and by $E(T)$ the set of edges of $T$. Every metric simplicial tree $T$ carries a natural path metric $d=d_T$. We equip $T$ with the metric topology, which is generally coarser than the simplicial topology; the two topologies agree if and only if the simplicial complex is locally finite \cite[Lemma~2.2.6]{Ch}. Any two points $x,y\in T$ are joined by a unique compact geodesic segment $[x,y]\subseteq T$ and between any two disjoint closed connected subsets $A,B\subset T$ there exists a unique compact connecting segment $[a,b]\subseteq T$ such that $A\cap [a,b]=a$ and $B\cap [a,b]=b$. In particular, $T$ is a simplicial $\mathbb{R}$-tree (see \cite{Ch} for an introduction to $\mathbb{R}$-trees) and, in fact, every simplicial $\mathbb{R}$-tree arises this way \cite[Theorem~2.2.10]{Ch}.

Let $G$ be a finitely generated group. A \emph{$G$-tree} is a metric simplicial tree $T$ on which $G$ acts by simplicial isometries without inversions of edges. Bass-Serre theory gives a correspondence between $G$-trees and metric graph of groups decompositions of $G$ (see \cite{Se}). We will always assume that the simplicial structure on $T$ is not a subdivision of a coarser simplicial structure with respect to which the action of $G$ on $T$ would still be simplicial and without inversions of edges (i.e., $T$ has no \emph{redundant} vertices). For a vertex or edge $x\in V(T)\cup E(T)$, we denote by $G_x\leq G$ its stabilizer. A group element $g\in G$ is \emph{elliptic in $T$} if it fixes a point in $T$ and \emph{hyperbolic} if not. The finite-order group elements of $G$ are always elliptic \cite[Proposition~19]{Se}. A finitely generated group that acts on a simplicial tree by simplicial automorphisms without inversions of edges has a global fixed point if and only if every group element is elliptic \cite[Corollary 6.5.3]{Se}.

A $G$-tree is \emph{minimal} if it does not contain a proper $G$-invariant subtree. Minimal $G$-trees are cocompact, i.e., their quotient graphs by the action of $G$ are finite \cite[Proposition~7.9]{Ba93}, and $G$-equivariant maps between minimal $G$-trees are always surjective; both properties will be used frequently and without further notice. The \emph{covolume} of a minimal $G$-tree $T$ is the volume of the finite metric quotient graph $G\backslash T$. There are 5 types of minimal $G$-trees (we adopt the naming convention from \cite{GL07}; see also \cite{CM87}): A minimal $G$-tree $T$ is \emph{trivial} if it is a point. It is \emph{dihedral} if it is a line and the action of $G$ does not preserve the orientation. The $G$-tree $T$ is \emph{linear abelian} if it is a line and $G$ acts by translations. It is \emph{genuine abelian} if $G$ fixes an end of $T$ and $T$ is not a line. Lastly, $T$ is \emph{irreducible} if $G$ contains a free subgroup of rank $2$ acting freely on $T$. In the following, we will almost exclusively be concerned with irreducible minimal $G$-trees, for reasons that will become apparent.

\subsubsection*{Translation lengths}

We briefly review well-known facts about translation lengths in $G$-trees. For further details see \cite{CM87} or \cite{Pa89}.

\begin{de}
Let $(T,d)$ be a $G$-tree. For a group element $g\in G$, define the \emph{translation length} of $g$ in $T$ by $$l(g)=l_T(g):=\inf_{x\in T}d(x,gx)\in\mathbb{R}_{\geq 0}$$ and its \emph{characteristic set} in $T$ by $C_g=C_T(g):=\left\{x\in T\ |\ d(x,gx)=l_T(g)\right\}\subseteq T.$
\end{de}

Conjugate group elements have the same translation length, and $C_g$ is always nonempty (i.e., $G$ acts on $T$ by semisimple isometries) and $g$-invariant. The translation length function $l_T\colon G\to\mathbb{R}$ defines a point in $\mathbb{R}^{\mathcal{C}(G)}$, where $\mathcal{C}(G)$ denotes the set of conjugacy classes of $G$. Clearly, $G$-equivariantly isometric $G$-trees have the same translation length function. If $T$ has finitely many $G$-orbits of edges, its translation length function has discrete image in $\mathbb{R}$.

A group element $g\in G$ is elliptic in $T$ if and only if $l(g)=0$. Its characteristic set is then its fixed point set and for all $x\in T$ the midpoint of the segment $[x,gx]$ is fixed by $g$. A group element $g\in G$ is hyperbolic in $T$ if and only if $l(g)>0$. Its characteristic set is then isometric to $\mathbb{R}$, the group element $g$ acts on $C_g$ by translations of length $l(g)$, and for all $k\in\mathbb{Z}\setminus\left\{0\right\}$ we have $l(g^k)=|k|\cdot l(g)$ and $C_{g^k}=C_g$. The characteristic set of a hyperbolic group element $g$ is the unique $g$-invariant line in $T$. We will instead denote it by $A_g$ and call it the \emph{hyperbolic axis} of $g$. Every $G$-tree without a global fixed point contains a unique minimal $G$-invariant subtree, given by the union of all hyperbolic axes \cite[Proposition~3.1]{CM87}.

\begin{prop}
\label{prop: translation formulas}
Let $T$ be a $G$-tree and $g,h\in G$.
\begin{enumerate}
\item\label{it: distance formula} For all $x\in T$ we have $d(x,gx)=l(g)+2 d(x,C_g)$.
\item\label{it: elliptics} Suppose that $g$ and $h$ are elliptic. Then $l(gh)=2d(C_g,C_h)$. In particular, if the fixed point sets of $g$ and $h$ are disjoint then $gh$ and $hg$ are hyperbolic.
\item\label{it: hyperbolics} Suppose that $g$ and $h$ are hyperbolic. If $A_g\cap A_h=\emptyset$ then $$l(gh)=l(hg)=l(g)+l(h)+2d(A_g,A_h)$$ and, in particular, $gh$ and $hg$ are hyperbolic. The hyperbolic axes of $gh$ and $hg$ then both intersect each $A_g$ and $A_h$.
\end{enumerate}
\end{prop}

\begin{proof}
See for example \cite[1.3]{CM87} and \cite[Propositions~1.6 and 1.8]{Pa89}.
\end{proof}

\section{Deformation spaces of $G$-trees}
\label{sec: deformation spaces}

In this section, we will review definitions, facts, and examples from the theory of deformation spaces of $G$-trees, mainly from \cite{Cl05} and \cite{GL07}.

Let $\mathcal{T}=\mathcal{T}(G)$ be the set of $G$-equivariant isometry classes of nontrivial minimal $G$-trees. We will always speak of ``$G$-trees'' in $\mathcal{T}$ and not of ``$G$-equivariant isometry classes of $G$-trees''.

\begin{de}
Given a $G$-tree $T\in\mathcal{T}$, a subgroup $H\leq G$ is an \textit{elliptic subgroup of $T$} if it fixes a point in $T$. We associate to $T$ its \textit{deformation space} $\mathcal{D}=\mathcal{D}(T)\subseteq\mathcal{T}$ consisting of all $G$-trees that have the same elliptic subgroups as $T$.
\end{de}

The finite subgroups of $G$ are elliptic in every $G$-tree \cite[Proposition~19]{Se}. If two $G$-trees $T,T'\in\mathcal{T}$ lie in the same deformation space then for all $g\in G$ we have $l_T(g)=0\Leftrightarrow l_{T'}(g)=0$. The converse, however, is not true, as an infinitely generated subgroup of $G$ all of whose elements fix a point in $T$ need not be elliptic; it then fixes a unique end of $T$ \cite[Proposition~3.4]{Ti70}.

An edge $e\in E(T)$ of a $G$-tree $T\in\mathcal{T}$ is \emph{collapsible} if its initial and terminal vertex $\iota(e)$ and $\tau(e)$ lie in distinct $G$-orbits and either $G_e=G_{\iota(e)}$ or $G_e=G_{\tau(e)}$. Collapsing all edges in the $G$-orbit of a collapsible edge is an \emph{elementary collapse}. An \emph{elementary expansion} is the reverse of an elementary collapse. A finite sequence of elementary collapses and expansions is an \emph{elementary deformation}. Two $G$-trees $T,T'\in\mathcal{T}$ have the same elliptic subgroups if and only if their underlying nonmetric $G$-trees are related by an elementary deformation \cite[Theorem 4.2]{Fo02}, and if and only if there exist $G$-equivariant (not necessarily simplicial) maps $T\to T'$ and $T'\to T$ \cite[Theorem 3.8]{GL07}.

\subsection{Topologies}
\label{sec: topologies}

Let $\mathcal{D}$ be a deformation space of $G$-trees. We consider three topologies on $\mathcal{D}$: The \emph{equivariant Gromov-Hausdorff topology} (see \cite{Pa89}); the \emph{axes topology}, which is the coarsest topology that makes the assignment of translation length functions $$l\colon\mathcal{D}\to\mathbb{R}^{\mathcal{C}(G)},\ T\mapsto l_T$$ continuous; and the \emph{weak topology}, which describes $\mathcal{D}$ as a union of open cones (see \cite{GL07}). Here, the \emph{open cone} spanned by a $G$-tree $T\in\mathcal{D}$ is the set of $G$-trees obtained by varying the lengths of the finitely many $G$-orbits of edges of $T$, while keeping them positive. Equivalently, it is the set of $G$-trees in $\mathcal{D}$ that are $G$-equivariantly homeomorphic to $T$. For a detailed discussion of the three topologies, including contractibility results, see \cite{Cl05} or \cite{GL07}.

\begin{de}
The multiplicative group of positive real numbers $\mathbb{R}_{>0}$ acts on $\mathcal{D}$ by scaling the metrics on the $G$-trees. The \textit{projectivized deformation space} $\mathcal{PD}$ is the quotient of $\mathcal{D}$ by this action, endowed with the quotient topology.
\end{de}

As a set, we will think of $\mathcal{PD}$ as the covolume-1-section in $\mathcal{D}$. In fact, if we endow $\mathcal{D}$ with the weak topology then the covolume function $\mathcal{D}\to (0,\infty),\ T\mapsto \mathrm{covol}(T)$ is continuous and the natural projection of the covolume-1-section in $\mathcal{D}$ to the projectivized deformation space $\mathcal{PD}$ is a homeomorphism.

When we equip $\mathcal{D}$ with the weak topology, the quotient $\mathcal{PD}$ inherits the structure of a simplicial complex with missing faces.

\subsubsection*{Further facts}

All $G$-trees in a given deformation space $\mathcal{D}$ have the same type (dihedral, linear abelian, genuine abelian, or irreducible; see Section~\ref{sec: trees}). If $\mathcal{D}$ consists of linear abelian or dihedral $G$-trees then the projectivized deformation space $\mathcal{PD}$ is a point \cite[Proposition 3.10]{GL07}. We say that $\mathcal{D}$ is \emph{genuine abelian} or \emph{irreducible} if the $G$-trees in $\mathcal{D}$ are genuine abelian or irreducible respectively, which are the only interesting cases.

The weak topology is finer than the equivariant Gromov-Hausdorff topology, which is finer than the axes topology. A weakly converging sequence also converges in the equivariant Gromov-Hausdorff topology and \emph{a fortiori} in the axes topology. The weak topology and the equivariant Gromov-Hausdorff topology agree on any finite union of open cones of $\mathcal{D}$ \cite[Proposition~5.2]{GL07}. The equivariant Gromov-Hausdorff topology and the axes topology agree if $\mathcal{D}$ is irreducible \cite{Pa89}.

Two irreducible minimal $G$-trees $T$ and $T'$ are $G$-equivariantly isometric if and only if for all $g\in G$ we have $l_T(g)=l_{T'}(g)$ \cite[Theorem~3.7]{CM87}. Therefore, if $\mathcal{D}$ is irreducible, the assignment of translation length functions $l\colon\mathcal{D}\to\mathbb{R}^{\mathcal{C}(G)},\ T\mapsto l_T$ is injective and the axes topology agrees with the subspace topology defined by this inclusion. In contrast, if $\mathcal{D}$ is a genuine abelian deformation space then all $G$-trees in $\mathcal{D}$ have the same translation length function up to scaling \cite[Proposition 3.10]{GL07}.

If some $G$-tree in $\mathcal{D}$ is locally finite then all $G$-trees in $\mathcal{D}$ are locally finite and we say that $\mathcal{D}$ is \emph{locally finite}. All vertex and edge stabilizers of any two $G$-trees in $\mathcal{D}$ are then commensurable as subgroups of $G$ and $\mathcal{PD}$ is a locally finite complex. If $\mathcal{D}$ consists of locally finite $G$-trees with finitely generated vertex stabilizers then the weak topology and the equivariant Gromov-Hausdorff topology agree on all of $\mathcal{D}$ \cite[Proposition 5.4]{GL07}.

\begin{ex}
\label{ex: Outer space}
Let $F_n$ be the free group of rank $n\geq 2$. The deformation space $\mathcal{X}_n$ of minimal $F_n$-trees that are acted on freely is locally finite and irreducible, and all three topologies agree on $\mathcal{X}_n$. The projectivized deformation space $\mathcal{PX}_n$ is better known as \emph{Culler-Vogtmann Outer space} \cite{CV86}.
\end{ex}

\begin{ex}
\label{ex: virtually free group}
More generally, let $G$ be a finitely generated virtually nonabelian free group, i.e., $G$ contains a finitely generated nonabelian free subgroup of finite index. It is a standard result that $G$ admits a minimal action on a simplicial tree $T$ with finite vertex (and edge) stabilizers \cite[Theorem~7.3]{SW79}. Since the finite subgroups of $G$ are elliptic in every $G$-tree, all minimal $G$-trees with finite vertex stabilizers lie in the same deformation space $\mathcal{D}$. The finite-index nonabelian free subgroup of $G$ must act freely on $T$, whence $\mathcal{D}$ is irreducible. The deformation space is locally finite and all three topologies agree on $\mathcal{D}$.
\end{ex}

\begin{ex}
\label{ex: nonelementary GBS groups}
If $G$ is a nonelementary GBS group (as defined in the introduction), all minimal $G$-trees with infinite cyclic vertex and edge stabilizers belong to the same deformation space $\mathcal{D}$ \cite[Corollary~6.10]{Fo02}, which is always locally finite. It is genuine abelian if $G$ is a solvable Baumslag-Solitar group $\BS(1,q)$ with $q\neq\pm 1$. In all other cases, it is irreducible and the three topologies agree on $\mathcal{D}$.
\end{ex}

\subsection{Action of the automorphism group}
\label{sec: automorphisms acting}

The automorphism group $\Aut(G)$ acts on $\mathcal{T}$ from the right by precomposing the $G$-actions on the trees. More precisely, given $T\in\mathcal{T}$ with isometric $G$-action $\rho\colon G\to \Isom(T)$ and $\Phi\in \Aut(G)$, we let $T\Phi$ be the $G$-tree with underlying metric simplicial tree $T$ and $G$-action $\rho\circ\Phi$. One easily sees that the normal subgroup of inner automorphisms $\Inn(G)\leq \Aut(G)$ acts trivially on $\mathcal{T}$ and we obtain an induced action of the outer automorphism group $\Out(G)=\Aut(G)/\Inn(G)$ on $\mathcal{T}$.

If $\phi\in \Out(G)$ leaves the set of elliptic subgroups of $T\in\mathcal{T}$ invariant then $T\phi$ lies in the same deformation space as $T$. In general, however, the twisted $G$-tree $T\phi\in\mathcal{T}$ might lie in a different deformation space.

\begin{de}
\label{de: Out D}
For a $G$-tree $T\in\mathcal{T}$ with associated deformation space $\mathcal{D}$, denote by $\Out_\mathcal{D}(G)\leq \Out(G)$ the subgroup of all automorphisms that leave the set of elliptic subgroups of $T$ invariant. The action of $\Out(G)$ on $\mathcal{T}$ restricts to an action of $\Out_\mathcal{D}(G)$ on $\mathcal{D}$.
\end{de}

\begin{prop}\label{prop: open cones to open cones}
The group $\Out_{\mathcal{D}}(G)$ acts on $\mathcal{D}$ by mapping open cones to open cones of the same dimension. For every $G$-tree $T\in\mathcal{D}$ only finitely many $G$-trees in the $\Out_{\mathcal{D}}(G)$-orbit of $T$ lie in the open cone spanned by $T$. The action of $\Out_\mathcal{D}(G)$ on $\mathcal{D}$ commutes with the action of $\mathbb{R}_{>0}$ and thus descends to an action on $\mathcal{PD}$.
\end{prop}

\begin{proof}
Let $T,T'\in\mathcal{D}$ and $\phi\in \Out_{\mathcal{D}}(G)$. If $T$ and $T'$ are $G$-equivariantly homeomorphic then $T\phi$ and $T'\phi$ are $G$-equivariantly homeomorphic as well, and $\Out_{\mathcal{D}}(G)$ acts on $\mathcal{D}$ by mapping open cones to open cones. Since $T$ and $T\phi$ have the same underlying metric simplicial tree, their open cones have the same dimension and the action of $\Out_\mathcal{D}(G)$ on $\mathcal{D}$ commutes with the action of $\mathbb{R}_{>0}$. In order to prove the second statement, suppose that $T$ and $T\phi$ are $G$-equivariantly homeomorphic. Then $T\phi$ is $G$-equivariantly isometric to $(T,d')$, where $d'$ is a metric on $T$ obtained by permuting the lengths of the $G$-orbits of edges of $T$, of which there are only finitely many.
\end{proof}

\subsubsection*{The modular homomorphism}

If $\mathcal{D}$ is a deformation space of locally finite $G$-trees, all vertex and edge stabilizers of all $G$-trees in $\mathcal{D}$ are commensurable as subgroups of $G$. We then define the \emph{modular homomorphism} $\mu=\mu(\mathcal{D})\colon G\to(\mathbb{Q}_{>0},\times)$ by $$\mu(g)=\frac{[H:(H\cap gHg^{-1})]}{[gHg^{-1}:(H\cap gHg^{-1})]}$$ where $H$ is any subgroup of $G$ commensurable with a vertex or edge stabilizer of some $G$-tree in $\mathcal{D}$. Indeed, $\mu$ does not depend on the choice of $H$. We say that $\mathcal{D}$ has \emph{no nontrivial integral modulus} if $\mathrm{im}(\mu)\cap\mathbb{Z}=\left\{1\right\}$.

\begin{lem}[{\cite[Lemma 2.4]{Le07}}]
\label{lem: GBS integral modulus}
Let $G$ be a nonelementary GBS group. The deformation space $\mathcal{D}$ of minimal $G$-trees with infinite cyclic vertex and edge stabilizers (Example~\ref{ex: nonelementary GBS groups}) has no nontrivial integral modulus if and only if $G$ contains no solvable Baumslag-Solitar group $\BS(1,n)$ with $n\geq 2$.
\end{lem}

\begin{rem}
The group $\BS(1,-n)$ contains a subgroup isomorphic to $\BS(1,n^2)$. Hence, if $G$ contains no solvable Baumslag-Solitar group $\BS(1,n)$ with $n\geq 2$ then it contains no solvable Baumslag-Solitar group $\BS(1,q)$ with $q\neq \pm 1$ and, in particular, the deformation space $\mathcal{D}$ is irreducible.
\end{rem}

A subgroup $H\leq G$ is \emph{small in $G$} (as in \cite[Section~8]{GL07}) if there does not exist a $G$-tree in which the axes of any two hyperbolic group elements of $H$ intersect in a compact set. Being small in $G$ is a commensurability invariant and stable under taking subgroups.

\begin{prop}[{\cite[Proposition 8.6]{GL07}}]
\label{prop: finitely many orbits of simplices}
Let $\mathcal{D}$ be a deformation space of locally finite irreducible $G$-trees whose vertex and edge stabilizers are all commensurable with a finitely generated subgroup $H\leq G$.
\begin{itemize}
\item[$(1)$] If $H$ is small in $G$ then $\Out_\mathcal{D}(G)=\Out(G)$.
\item[$(2)$] If every subgroup of $G$ commensurable with $H$ has finite outer automorphism group and $\mathcal{D}$ has no nontrivial integral modulus then $\Out_\mathcal{D}(G)$ acts on $\mathcal{D}$ with finitely many orbits of open cones (and on the projectivized deformation space $\mathcal{PD}$ with finitely many orbits of open simplices).
\end{itemize}
\end{prop}

\begin{ex}
The unprojectivized Outer space $\mathcal{X}_n$ (Example~\ref{ex: Outer space}) is locally finite and irreducible, and all vertex and edge stabilizers of the $F_n$-trees in $\mathcal{X}_n$ are trivial. We clearly have $\Out_{\mathcal{X}_n}(F_n)=\Out(F_n)$, and $\Out(F_n)$ acts on Outer space $\mathcal{PX}_n$ with finitely many orbits of simplices.
\end{ex}

\begin{ex}
\label{ex: virtually free}
More generally, let $G$ be a finitely generated virtually nonabelian free group. The deformation space $\mathcal{D}$ of minimal $G$-trees with finite vertex stabilizers is locally finite and irreducible (Example~\ref{ex: virtually free group}). Choosing $H=\left\{1\right\}$, we see that $\mu(\mathcal{D})\equiv 1$. We have $\Out_\mathcal{D}(G)=\Out(G)$ and $\Out(G)$ acts on $\mathcal{PD}$ with finitely many orbits of simplices.
\end{ex}

\begin{ex}
\label{ex: GBS orbits}
Let $G$ be a nonelementary GBS group that contains no solvable Baumslag-Solitar group $\BS(1,n)$ with $n\geq 2$. The deformation space $\mathcal{D}$ of minimal $G$-trees with infinite cyclic vertex and edge stabilizers is locally finite and irreducible (Example~\ref{ex: nonelementary GBS groups}). Let $H$ be any vertex or edge stabilizer of any $G$-tree in $\mathcal{D}$. If $G$ acts on a tree such that $H$ acts hyperbolically then all nontrivial elements of $H$ have the same hyperbolic axis (because $H$ is infinite cyclic), whence $H$ is small in $G$ and $\Out_\mathcal{D}(G)=\Out(G)$. Every subgroup of $G$ commensurable with $H$, being virtually cyclic, has finite outer automorphism group. By Lemma~\ref{lem: GBS integral modulus}, $\mathcal{D}$ has no nontrivial integral modulus and hence $\Out(G)$ acts on $\mathcal{PD}$ with finitely many orbits of simplices.
\end{ex}

\section{The Lipschitz metric}
\label{sec: Lipschitz metric}

Let $\mathcal{D}$ be a deformation space of $G$-trees and $T,T'\in\mathcal{D}$. As $T$ and $T'$ lie in the same deformation space, there exists a $G$-equivariant map $f\colon T\to T'$, which we may choose to be Lipschitz continuous. We denote by $\sigma(f)$ its Lipschitz constant.

Every $G$-equivariant Lipschitz map $f\colon T\to T'$ is $G$-equivariantly homotopic relative to the vertices of $T$ to a $G$-equivariant Lipschitz map $f'\colon T\to T'$ that is linear (i.e., either constant or an immersion with constant slope) on edges. The Lipschitz constant $\sigma(f')$ is then given by the maximal slope of $f'$ on the finitely many $G$-orbits of edges of $T$ and we have $\sigma(f')\leq\sigma(f)$. We may therefore always assume every $G$-equivariant Lipschitz map $f\colon T\to T'$ to be linear on edges without increasing its Lipschitz constant.

\begin{de}
Define $\sigma(T,T')=\inf_f \sigma(f)$, where $f$ ranges over all $G$-equivariant Lipschitz maps from $T$ to $T'$.
\end{de}

Recall that, as a set, we think of the projectivized deformation space $\mathcal{PD}$ as the covolume-1-section in $\mathcal{D}$. With this convention, we can assign to each pair of projectivized $G$-trees $(T,T')\in\mathcal{PD}\times\mathcal{PD}$ the well-defined value $\sigma(T,T')$.

\begin{deprop}
\label{prop: Lipschitz metric}
The function $$d_{Lip}\colon\mathcal{PD}\times\mathcal{PD}\to\mathbb{R},\ (T,T')\mapsto\log\left(\sigma(T,T')\right)$$ is an asymmetric pseudometric on $\mathcal{PD}$. That is, for all $T,T',T''\in\mathcal{PD}$ we have
\begin{enumerate}
\item $d_{Lip}(T,T')\geq 0$;
\item\label{it: distance zero} if $T$ and $T'$ are $G$-equivariantly isometric then $d_{Lip}(T,T')=0$;
\item $d_{Lip}(T,T'')\leq d_{Lip}(T,T')+d_{Lip}(T',T'')$.
\end{enumerate}
We call $d_{Lip}$ the \emph{Lipschitz metric}.
\end{deprop}

\begin{proof}
To prove $(1)$, let $f\colon T\to T'$ be a $G$-equivariant Lipschitz map. We will show that $\sigma(f)$ is bounded below by 1. Since the $G$-trees $T$ and $T'$ are minimal, both $f$ and the induced map on quotient graphs $G\backslash f\colon G\backslash T\to G\backslash T'$ are surjective. We have $\sigma(G\backslash f)=\sigma(f)$ and $\mathrm{vol}(G\backslash T)=\mathrm{vol}(G\backslash T')=1$. If now $\sigma(f)<1$ then $$\mathrm{vol}(\mathrm{im}(G\backslash f))\leq\sigma(f)\cdot \mathrm{vol}(G\backslash T)<1$$ contradicting the surjectivity of $G\backslash f$.

Statement $(2)$ is immediate. In order to show $(3)$, observe that for any sequence of $G$-equivariant Lipschitz maps $T\stackrel{f}{\to}T'\stackrel{f'}{\to}T''$ we have $\sigma(T,T'')\leq\sigma(f'\circ f)$ and $\sigma(f'\circ f)\leq\sigma(f)\cdot\sigma(f')$, whence
\begin{align*}
\log\left(\sigma(T,T'')\right) \leq \inf_{f,f'}\log\left(\sigma(f'\circ f)\right) &\leq \inf_{f,f'}\log\left(\sigma(f)\cdot \sigma(f')\right)\\
&= \inf_{f,f'}\left(\log\left(\sigma(f)\right)+\log\left(\sigma(f')\right)\right)\\
&= \inf_f\log\left(\sigma(f)\right) + \inf_{f'}\log\left(\sigma(f')\right)\\
&= \log\left(\sigma(T,T')\right)+\log\left(\sigma(T',T'')\right).\qedhere
\end{align*}
\end{proof}

\begin{prop} 
The group $\Out_{\mathcal{D}}(G)$ acts on $(\mathcal{PD},d_{Lip})$ by isometries, i.e., for all $T,T'\in\mathcal{PD}$ and $\phi\in \Out_{\mathcal{D}}(G)$ we have $d_{Lip}(T\phi,T'\phi)=d_{Lip}(T,T')$.
\end{prop}

\begin{proof}
Every $G$-equivariant map from $T$ to $T'$ is also $G$-equivariant with respect to the actions twisted along $\phi$, and vice versa.
\end{proof}

The following example demonstrates why we speak of the Lipschitz metric as an ``asymmetric pseudometric'':

\begin{ex}
\label{ex: not an asymmetric metric}
In general we have $d_{Lip}(T,T')\neq d_{Lip}(T',T)$ (see \cite{AB12} for examples in the special case of Outer space; see also the remark after Proposition~\ref{prop: symmetrized metric}). Moreover, $d_{Lip}(T,T')=0$ does not generally imply that $T$ and $T'$ are $G$-equivariantly isometric (see Proposition~\ref{prop: asymmetric metric} for an exception in the case of Outer space; see also Section~\ref{sec: elliptics}):

Let $G=F_2\ast \left(\mathbb{Z}/2\mathbb{Z}\right)$ and consider the graph of groups decompositions $\Gamma$ and $\Gamma'$ of $G$ as in Figure \ref{fig: pseudometric}, where all edge group inclusions are the obvious ones and all edges have constant length $\frac{1}{3}$.
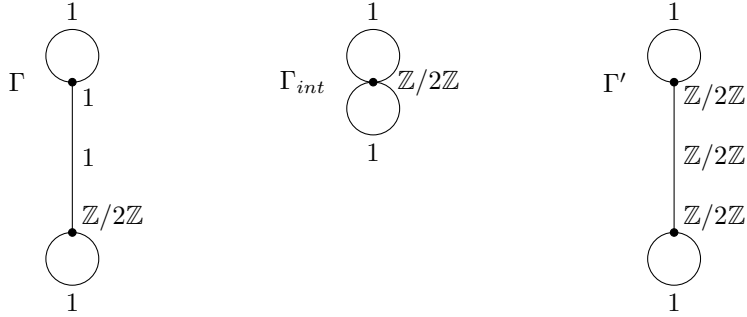
\begin{figure}[h!]
$$\begin{tikzpicture}
\draw [-] (1,.35) circle (10pt);
\draw [-] (5,.35) circle (10pt);
\draw [-] (9,.35) circle (10pt);
\draw [-] (1,-2.35) circle (10pt);
\draw [-] (5,-.35) circle (10pt);
\draw [-] (9,-2.35) circle (10pt);

\node [left] at (.5,0) {$\Gamma$};
\node [left] at (4.5,0) {$\Gamma_{int}$};
\node [left] at (8.5,0) {$\Gamma'$};

\node [above] at (1,.7) {1};
\node [right] at (1,-.2) {1};
\node [right] at (1,-1) {1};
\node [right] at (1,-1.8) {$\mathbb{Z}/2\mathbb{Z}$};
\node [below] at (1,-2.7) {1};

\node [above] at (5,.7) {1};
\node [right] at (5.2,0) {$\mathbb{Z}/2\mathbb{Z}$};
\node [below] at (5,-.7) {1};

\node [above] at (9,.7) {1};
\node [right] at (9,-.2) {$\mathbb{Z}/2\mathbb{Z}$};
\node [right] at (9,-1) {$\mathbb{Z}/2\mathbb{Z}$};
\node [right] at (9,-1.8) {$\mathbb{Z}/2\mathbb{Z}$};
\node [below] at (9,-2.7) {1};

\draw [-] (1,0) -- (1,-2);
\draw [-] (9,0) -- (9,-2);

\draw [fill] (1,0) circle [radius=.05];
\draw [fill] (1,-2) circle [radius=.05];
\draw [fill] (5,0) circle [radius=.05];
\draw [fill] (9,-2) circle [radius=.05];
\draw [fill] (9,0) circle [radius=.05];
\end{tikzpicture}$$
\caption{The Bass-Serre trees of the above graphs of groups lie in the same deformation space of $G$-trees. They are irreducible and locally finite. Since the group $G$ acts on them cocompactly and with finite point stabilizers, it is virtually free.}
\label{fig: pseudometric}
\end{figure}
The corresponding Bass-Serre trees $T$ and $T'$ lie in the same deformation space, as they are related by an elementary collapse followed by an elementary expansion (the intermediate graph of groups is given by $\Gamma_{int}$). The vertices of $T$ have valence 3 and 6, whereas the vertices of $T'$ all have valence 5. Hence, $T$ and $T'$ are not homeomorphic and in particular not $G$-equivariantly isometric. Still, the natural morphism of graphs of groups (in the sense of \cite{Ba93}) from $\Gamma$ to $\Gamma'$ lifts to a $G$-equivariant map from $T$ to $T'$ that is an isometry on edges and thus has Lipschitz constant 1, whence $d_{Lip}(T,T')=0$.
\end{ex}

\subsubsection*{The symmetrized Lipschitz metric}

A standard way to overcome these issues is to consider the symmetrized Lipschitz metric $$d_{Lip}^{sym}\colon\mathcal{PD}\times\mathcal{PD}\to\mathbb{R},\ (T,T')\mapsto d_{Lip}(T,T')+d_{Lip}(T',T)$$
which turns out to be an actual metric on projectivized deformation spaces of irreducible $G$-trees (in Section~\ref{sec: convergent sequences} we discuss its convergent sequences):

\begin{prop}
\label{prop: symmetrized metric}
If the projectivized deformation space $\mathcal{PD}$ consists of irreducible $G$-trees then for all $T,T'\in\mathcal{PD}$ we have $d_{Lip}^{sym}(T,T')=0$ if and only if $T$ and $T'$ are $G$-equivariantly isometric.
\end{prop}

\begin{proof}
By Proposition~\ref{prop: Lipschitz metric}(\ref{it: distance zero}) it suffices to show the ``only if'' direction. Suppose that we have $d_{Lip}^{sym}(T,T')=0$, equivalently $d_{Lip}(T,T')=0$ and $d_{Lip}(T',T)=0$. Then for all $\varepsilon>0$ there exist $G$-equivariant $(1+\varepsilon)$-Lipschitz maps $f\colon T\to T'$ and $f'\colon T'\to T$. Let $g\in G$ be a hyperbolic group element in $T$ and $p\in A_g\subset T$ a point in its hyperbolic axis. We have
\begin{align*}
l_{T'}(g)\leq d(f(p),gf(p)) &= d(f(p),f(gp))\\
&\leq \sigma(f)\cdot d(p,gp)=\sigma(f)\cdot l_T(g)\leq (1+\varepsilon)\cdot l_T(g)
\end{align*}
and, analogously, $l_T(g)\leq (1+\varepsilon)\cdot l_{T'}(g)$. As $\varepsilon$ was arbitrary, we conclude that $l_T=l_{T'}$ and hence, by \cite[Theorem~3.7]{CM87}, that the irreducible $G$-trees $T$ and $T'$ are $G$-equivariantly isometric.
\end{proof}

\begin{rem}
Thus, for $T$ and $T'$ as in Example~\ref{ex: not an asymmetric metric} we have $d_{Lip}(T',T)>0$, since $d_{Lip}(T,T')=0$ but $T$ and $T'$ are not $G$-equivariantly isometric.
\end{rem}

Nevertheless, the arguments in Section \ref{sec: classification} are specific for the asymmetric pseudometric $d_{Lip}$. Besides, in contrast to $d_{Lip}$, the symmetrization $d_{Lip}^{sym}$ fails to be geodesic, as was shown in \cite[Section~6]{FM11} in the special case of Outer space (see Section \ref{sec: geodesics} for the existence of $d_{Lip}$-geodesics).

\subsection{Minimal stretch maps}

\begin{thm}[Existence of minimal stretch maps]
\label{thm: irreducible infimum realized}
Let $\mathcal{D}$ be a deformation space of irreducible $G$-trees. For all $T,T'\in\mathcal{D}$ there exists a $G$-equivariant Lipschitz map $f\colon T\to T'$ such that $\sigma(f)=\sigma(T,T')$.
\end{thm}

The proof of Theorem~\ref{thm: irreducible infimum realized} will involve an argument of Horbez \cite{Ho14} that uses nonprincipal ultrafilters and ultralimits of metric spaces, which are defined as follows:

\begin{de}
A \emph{nonprincipal ultrafilter} $\omega$ on an infinite set $I$ is a finitely additive probability measure with values in $\left\{0,1\right\}$ such that all subsets $S\subseteq I$ are $\omega$-measurable and $\omega(S)=0$ if $S$ is finite.
\end{de}

Existence of nonprincipal ultrafilters follows from the axiom of choice. Given a nonprincipal ultrafilter $\omega$ on the set of natural numbers $\mathbb{N}$, for every bounded sequence $(c_n)_{n\in\mathbb{N}}\subset\mathbb{R}$ there exists a unique point $\lim_\omega c_n\in\mathbb{R}$ such that for all open neighborhoods $U$ of $\lim_\omega c_n$ we have $\omega(\left\{n\in\mathbb{N}\ |\ c_n\in U\right\})=1$ (see, for instance, \cite[9.1]{Ka}). In particular, if the sequence $(c_n)_{n\in\mathbb{N}}$ converges then $\lim_\omega c_n=\lim_{n\to\infty}c_n$.

\begin{de}
Let $\omega$ be a nonprincipal ultrafilter on $\mathbb{N}$. For a sequence of metric spaces $(X_n,d_n)_{n\in\mathbb{N}}$ with basepoints $(p_n)_{n\in\mathbb{N}}$ let $X_\infty$ be the set of all sequences $(x_n)_{n\in\mathbb{N}}\in\prod_{n\in\mathbb{N}}{X_n}$ for which the sequence $(d_n(x_n,p_n))_{n\in\mathbb{N}}\subset\mathbb{R}$ is bounded. Let $\sim$ be the equivalence relation on $X_\infty$ defined by $$(x_n)_{n\in\mathbb{N}}\sim (y_n)_{n\in\mathbb{N}}\quad\text{if}\quad\lim_\omega d_n(x_n,y_n)=0.$$ Define the \emph{$\omega$-ultralimit} $X_\omega$ of $(X_n,d_n,p_n)_{n\in\mathbb{N}}$ as the quotient $X_\infty/\sim$ endowed with the metric $d_\omega((x_n)_{n\in\mathbb{N}},(y_n)_{n\in\mathbb{N}})=\lim_\omega d_n(x_n,y_n)$.
\end{de}

If each $(X_n,d_n),\ n\in\mathbb{N}$ is a complete $\mathbb{R}$-tree then $(X_\omega,d_\omega)$ is again a complete $\mathbb{R}$-tree \cite[Lemma~4.6]{St09}. Moreover, if a group $G$ acts on each $(X_n,d_n)$ by isometries and for all $g\in G$ the sequence $(d_n(gp_n,p_n))_{n\in\mathbb{N}}\subset\mathbb{R}$ is bounded then $(X_\omega,d_\omega)$ carries a natural isometric $G$-action: For $g\in G$ and $(x_n)_{n\in\mathbb{N}}\in X_\omega$ we define $g(x_n)_{n\in\mathbb{N}}=(gx_n)_{n\in\mathbb{N}}$. Since for all $g\in G$ and $n\in\mathbb{N}$ we have
\begin{align*}
d_n(gx_n,p_n) &\leq d_n(gx_n,gp_n)+d_n(gp_n,p_n)\\
&=d_n(x_n,p_n)+d_n(gp_n,p_n)
\end{align*}
and the sequences of real numbers $(d_n(x_n,p_n))_{n\in\mathbb{N}}$ and $(d_n(gp_n,p_n))_{n\in\mathbb{N}}$ are bounded, so is the sequence $(d_n(gx_n,p_n))_{n\in\mathbb{N}}$. For all $g\in G$ we have
\begin{align*}
d_\omega(g(x_n)_{n\in\mathbb{N}},g(y_n)_{n\in\mathbb{N}}) &= \lim_\omega d_n(gx_n,gy_n)\\
&=\lim_\omega d_n(x_n,y_n)=d_\omega((x_n)_{n\in\mathbb{N}},(y_n)_{n\in\mathbb{N}})
\end{align*}
and the action of $G$ on $(X_\omega,d_\omega)$ is by isometries.

\begin{proof}[Proof of Theorem~\ref{thm: irreducible infimum realized}]
Let $T,T'\in\mathcal{D}$ and $C=\sigma(T,T')$. We wish to construct a $G$-equivariant $C$-Lipschitz map $f\colon T\to T'$.

Let $\omega$ be a nonprincipal ultrafilter on $\mathbb{N}$ and $(f_n\colon T\to T')_{n\in\mathbb{N}}$ a sequence of $G$-equivariant $C_n$-Lipschitz maps with $C_n\leq 2C$ and $\lim_{n\to\infty}C_n=C$. We will first choose a distinguished basepoint $p\in T$ and show that for all $n\in\mathbb{N}$ the image $f_n(p)$ lies in a bounded subset of $T'$ that does not depend on $n$. This then implies that for all $g\in G$ the sequence $d'(gf_n(p),f_n(p))_{n\in\mathbb{N}}\subset\mathbb{R}$ is bounded and that the $\omega$-ultralimit $T'_\omega=(T',d',f_n(p))_{\omega}$ carries a natural isometric $G$-action. (Evidently, $T_\omega=(T,d,p)_{\omega}$ carries a natural isometric $G$-action as well.) Indeed, as the action of $G$ on $T$ is irreducible, $G$ contains a free subgroup of rank 2 acting freely. Suppose that this free subgroup is generated by $g,h\in G$. Since $T$ and $T'$ have the same elliptic subgroups, the free subgroup $\langle g,h\rangle\leq G$ also acts freely on $T'$. If the hyperbolic axes $A_g$ and $A_h$ in $T$ intersect, they must intersect in a compact segment, as we could otherwise find integers $k,l\in\mathbb{Z}\setminus\left\{0\right\}$ such that $g^kh^{-l}$ fixes a point in $A_g\cap A_h$. For the following arguments we will assume that they intersect; if they are disjoint, we replace the basis of the free subgroup with $\left\{g,hg\right\}$, whose associated axes then intersect by Proposition~\ref{prop: translation formulas}(\ref{it: hyperbolics}). Let $p\in A_g\cap A_h$ be a point that lies in both axes and denote the hyperbolic axes of $g$ and $h$ in $T'$ by $A'_g$ and $A'_h$ respectively. By Proposition~\ref{prop: translation formulas}(\ref{it: distance formula}), and since $f_n$ is $G$-equivariant and $C_n$-Lipschitz with $C_n\leq 2C$, for all $n\in\mathbb{N}$ we have
\begin{align*}2C\cdot l_T(g)=2C\cdot d(gp,p) &\geq d'(f_n(gp),f_n(p))\\
&= d'(gf_n(p),f_n(p))=l_{T'}(g)+2d'(f_n(p),A'_g)
\end{align*}
and hence $d'(f_n(p),A'_g)\leq\frac{1}{2}(2C\cdot l_T(g)-l_{T'}(g))\leq C\cdot l_T(g)$. Thus, $f_n(p)$ lies within a $(C,l_T(g))$-bounded\footnote{Bounded in terms of $C$ and $l_T(g)$.} distance from $A'_g$ and, analogously, within a $(C,l_T(h))$-bounded distance from $A'_h$. We conclude that $f_n(p)$ lies within a $(C,l_T(g),l_T(h))$-bounded distance from the compact segment $A'_g\cap A'_h$ if the two axes intersect and from the unique compact connecting segment between them if they are disjoint. In particular, $f_n(p)$ lies in a bounded subset of $T'$ that does not depend on $n$. As remarked above, this implies that the ultralimits $T_\omega=(T,d,p)_{\omega}$ and $T'_\omega=(T',d',f_n(p))_{\omega}$ carry natural isometric $G$-actions.

The $G$-trees $T$ and $T'$ naturally embed $G$-equivariantly and isometrically into $T_\omega$ and $T'_\omega$ respectively: Since for all $n\in\mathbb{N}$ the point $f_n(p)\in T'$ lies in a bounded subset that does not depend on $n$, for all $x\in T'$ the sequence $(d'(x,f_n(p)))_{n\in\mathbb{N}}$ is bounded and hence the constant sequence $(x)_{n\in\mathbb{N}}$ defines a point in $T'_\omega$. One easily verifies that the natural inclusion $T'\hookrightarrow T'_\omega,\ x\mapsto (x)_{n\in\mathbb{N}}$ is indeed $G$-equivariant and isometric. We analogously obtain a $G$-equivariant isometric embedding $T\hookrightarrow T_\omega$.

Observe next that if $(d(x_n,p))_{n\in\mathbb{N}}$ is bounded then $(d'(f_n(x_n),f_n(p)))_{n\in\mathbb{N}}$ is bounded as well, since for all $n\in\mathbb{N}$ we have $d'(f_n(x_n),f_n(p))\leq 2C\cdot d(x_n,p)$. Thus, the maps $(f_n)_{n\in\mathbb{N}}$ induce a natural map $$f_\omega\colon T_\omega\to T'_\omega,\ (x_n)_{n\in\mathbb{N}}\mapsto (f_n(x_n))_{n\in\mathbb{N}}.$$
The map $f_\omega$ is easily seen to be $G$-equivariant, since for all $g\in G$ we have
\begin{align*}
f_\omega(g(x_n)_{n\in\mathbb{N}})=f_\omega((gx_n)_{n\in\mathbb{N}}) &= (f_n(gx_n))_{n\in\mathbb{N}}\\
&= (gf_n(x_n))_{n\in\mathbb{N}}=g(f_n(x_n))_{n\in\mathbb{N}}=gf_\omega((x_n)_{n\in\mathbb{N}}).
\end{align*}
Moreover, $f_\omega$ is $C$-Lipschitz, since for all $(x_n)_{n\in\mathbb{N}}, (y_n)_{n\in\mathbb{N}}\in T_\omega$ we have
\begin{align*}
d'_\omega(f_\omega((x_n)_{n\in\mathbb{N}}),f_\omega((y_n)_{n\in\mathbb{N}})) &= \lim_{\omega} d'(f_n(x_n),f_n(y_n))\\
&\leq \lim_{\omega} \left(C_n\cdot d(x_n,y_n)\right)\\
&= \lim_\omega C_n \cdot \lim_\omega d(x_n,y_n)
= C\cdot d_\omega((x_n)_{n\in\mathbb{N}},(y_n)_{n\in\mathbb{N}}).
\end{align*}

Finally, $T'_\omega$ is a complete $\mathbb{R}$-tree, being the $\omega$-ultralimit of complete $\mathbb{R}$-trees (namely, metric simplicial trees). In particular, the metric simplicial tree $T'$ embeds into $T'_\omega$ as a closed subspace, as complete subspaces of complete metric spaces are closed. By the nature of $\mathbb{R}$-trees, there exists a continuous nearest point projection of $T'_\omega$ onto the closed $G$-invariant subtree $T'$, which is easily seen to be $G$-equivariant and 1-Lipschitz. We define $f\colon T\to T'$ as the composition of the $G$-equivariant isometric embedding $T\hookrightarrow T_\omega$ with $f_\omega\colon T_\omega\to T'_\omega$ and the nearest point projection $T'_\omega\to T'$, and we obtain a $G$-equivariant $C$-Lipschitz map from $T$ to $T'$.
\end{proof}

\subsubsection*{Train tracks and optimal maps}

In fact, we will be interested in particularly nice $G$-equivariant Lipschitz maps realizing $\sigma(T,T')$, so-called \emph{optimal maps}. In order to define and construct optimal maps, we involve the concept of \emph{train tracks}:

\begin{de}
\label{de: train tracks}
Let $\mathcal{D}$ be a deformation space of $G$-trees and $T\in\mathcal{D}$. A \emph{direction} at a point $x\in T$ is a germ of isometric embeddings $\gamma\colon[0,\varepsilon)\to T,\ \varepsilon>0$ with $\gamma(0)=x$. Given $g\in G$ with $gx\neq x$, we will denote the unique direction at $x$ pointing towards $gx$ by $\delta_{x,gx}$. Denote the set of directions at $x$ by $D_x T$. A \emph{train track structure} on $T$ is a collection of equivalence relations, one on $D_v T$ for each vertex $v\in V(T)$, such that two directions $\delta_1,\delta_2\in D_v T$ are equivalent (denoted $\delta_1\sim\delta_2$) if and only if for all $g\in G$ the directions $g\delta_1, g\delta_2\in D_{gv}T$ are equivalent as well. Equivalence classes of directions at a vertex $v\in V(T)$ are called \emph{gates} at $v$. A \emph{turn} at a vertex $v\in V(T)$ is a pair of directions at $v$. Given a train track structure on $T$, we say that a turn at a vertex is \emph{illegal} if the two directions are equivalent, i.e., if they represent the same gate, and \emph{legal} if not. Whenever a nondegenerate immersed path $\gamma$ in $T$ passes through a vertex $v$ of $T$, we may locally reparametrize $\gamma$ to an isometric embedding so that the incoming direction (with opposite orientation) and the outgoing direction of $\gamma$ at $v$ define a turn at $v$. A nondegenerate immersed path in $T$ is \emph{legal} if it only makes legal turns and \emph{illegal} otherwise.
\end{de}

\begin{de}
\label{de: tension forest}
Let $\mathcal{D}$ be a deformation space of $G$-trees and $T,T'\in\mathcal{D}$. Let $f\colon T\to T'$ be a $G$-equivariant map that is linear on edges. We denote the union of all (closed) edges of $T$ on which $f$ attains its maximal slope by $\Delta(f)\subset T$ and we call it the \emph{tension forest} of $f$. The tension forest $\Delta(f)\subset T$ is a $G$-invariant subforest.
\end{de}

Every $G$-equivariant map $f\colon T\to T'$ that is linear on edges defines a natural train track structure on its tension forest $\Delta=\Delta(f)\subset T$ as follows: For each vertex $v\in V(\Delta)$ we have a map $D_vf\colon D_v\Delta\to D_{f(v)}{T'}$ that maps the direction of $\gamma\colon [0,\varepsilon)\to\Delta$ with $\gamma(0)=v$ to the direction of the unique isometric embedding in the reparametrization class of $f\circ\gamma$ (since $f$ does not collapse any edges in its tension forest, it has nonzero slope on the image of $\gamma$). We define two directions $\delta_1,\delta_2\in D_v\Delta$ to be equivalent if $D_vf(\delta_1)=D_vf(\delta_2)$. By the $G$-equivariance of $f$, this collection of equivalence relations is indeed a train track structure on $\Delta$.

The tension forest $\Delta(f)$ endowed with the train track structure defined by $f$ might have vertices of valence 1 and, more generally, there might be vertices with only one gate.

\begin{de}
A $G$-equivariant Lipschitz map $f\colon T\to T'$ that realizes $\sigma(T,T')$ and is linear on edges is an \emph{optimal map} if its tension forest $\Delta(f)$ has at least 2 gates at every vertex.
\end{de}

Optimality of $f$ implies that any legal path in $\Delta(f)$ may be extended in both directions to a longer legal path and, inductively, that there exists a legal line in $\Delta(f)$. This will be made use of in the proof of Theorem~\ref{thm: inf = sup}.

\begin{prop}\label{prop: existence of optimal maps}
Let $\mathcal{D}$ be a deformation space of $G$-trees and $T,T'\in\mathcal{D}$. Every $G$-equivariant Lipschitz map $f\colon T\to T'$ that realizes $\sigma(T,T')$ and is linear on edges is $G$-equivariantly homotopic to an optimal map $f'\colon T\to T'$ with $\Delta(f')\subseteq\Delta(f)$. If $f$ is not optimal to begin with then we have $\Delta(f')\neq\Delta(f)$.
\end{prop}

Theorem~\ref{thm: irreducible infimum realized} and Proposition~\ref{prop: existence of optimal maps} imply that if $\mathcal{D}$ is an irreducible deformation space then for all $T,T'\in\mathcal{D}$ there exists an optimal map $f\colon T\to T'$.

\begin{proof}
Let $\Delta=\Delta(f)$. If a vertex $v\in V(\Delta)$ has only one gate $\delta\in D_v\Delta$, slightly move $f(v)$ in the direction of $D_vf(\delta)\in D_{f(v)}T'$ (see Figure~\ref{fig: move f(v)}).
\begin{figure}[h!]
$$\begin{tikzpicture}
\draw [-] (0,-1) -- (6,-2.5);
\draw [-] (1,-3) -- (6,-2.5);
\draw [-] (5,-1) -- (6,-2.5);
\draw[dashed] [-] [line width=0.075cm] (6,-2.5) -- (10,-4);
\draw[dashed] [-] [line width=0.075cm] (10,-3) -- (8,-3.25);
\draw[dashed] [-] [line width=0.075cm] (8,-4) -- (7,-2.87);
\draw [->] [line width=0.075cm] (6,-2.5) -- (6.66,-2.75);
\draw [-] (6,-2.5) -- (10,-4);
\draw [-] (10,-3) -- (8,-3.25);
\draw [-] (8,-4) -- (7,-2.87);
\node [below] at (5.8,-2.6) {$f(v)$};
\draw [fill] (0,-1) circle [radius=.075];
\draw [fill] (1,-3) circle [radius=.075];
\draw [fill] (5,-1) circle [radius=.075];
\draw [fill] (6,-2.5) circle [radius=.075];
\draw [fill] (8,-4) circle [radius=.075];
\draw [fill] (10,-3) circle [radius=.075];
\draw [fill] (10,-4) circle [radius=.075];
\end{tikzpicture}$$
\caption{The image of $\Delta$ under $f$ (dashed) and the direction in which we slightly move $f(v)$ (arrow).}
\label{fig: move f(v)}
\end{figure}
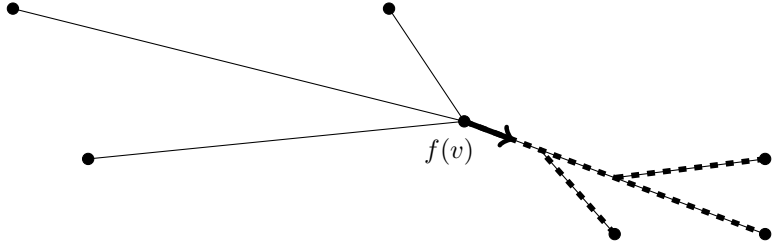
Perform this perturbation $G$-equivariantly and keep the homotopy fixed on all other $G$-orbits of vertices of $T$. This decreases the slope of $f$ on the $G$-orbits of all edges of $\Delta$ adjacent to $v$ and we obtain a $G$-equivariant Lipschitz map $f'\colon T\to T'$ with $\Delta(f')\subset \Delta$ but $\Delta(f')\neq\Delta$. Keeping the perturbation small enough ensures that the (finitely many) $G$-orbits of the edges of $T\setminus\Delta$ adjacent to $v$, on which the slope is increased, do not become part of the new tension forest. As $f$ is assumed to have minimal Lipschitz constant among all $G$-equivariant Lipschitz maps from $T$ to $T'$, we will not have removed all edges of $\Delta$ and started over with a new tension forest that corresponds to a strictly smaller maximal stretching factor. This process eventually terminates by the cocompactness of $T$.
\end{proof}

\subsection{Witnesses}

The results in this section will imply that the Lipschitz metric on projectivized deformation space of irreducible $G$-trees can be computed in terms of hyperbolic translation lengths. We begin with an easy observation:

\begin{lem}
\label{lem: pre-witness}
Let $\mathcal{D}$ be a deformation space of $G$-trees and $T,T'\in\mathcal{D}$. For any $G$-equivariant Lipschitz map $f\colon T\to T'$ and any hyperbolic group element $g\in G$ we have $\sigma(f)\geq\frac{l_{T'}(g)}{l_T(g)}.$ In particular, we have $\sigma(T,T')\geq\sup_g\frac{l_{T'}(g)}{l_T(g)}$, where $g$ ranges over all hyperbolic group elements of $G$.
\end{lem}

\begin{proof} Let $p\in A_g$. We have $l_{T'}(g) \leq d(gf(p),f(p))\leq \sigma(f)\cdot d(gp,p) = \sigma(f)\cdot l_T(g)$, whence the claim.
\end{proof}

\begin{thm}[Existence of witnesses]
\label{thm: inf = sup}
Let $\mathcal{D}$ be a deformation space of irreducible $G$-trees. For all $T,T'\in\mathcal{D}$ there exists a hyperbolic group element $\xi\in G$ such that $$\sigma(T,T')=\frac{l_{T'}(\xi)}{l_T(\xi)}=\sup_g\frac{l_{T'}(g)}{l_T(g)}$$ where $g$ ranges over all hyperbolic group elements of $G$.  In fact, we can always arrange that some (and hence any) fundamental domain for the action of $\xi$ on its hyperbolic axis $A_\xi\subset T$ meets each $G$-orbit of vertices of $T$ at most 10 times.
\end{thm}

We will call a hyperbolic group element $\xi\in G$ (or, depending on the context, its hyperbolic axis $A_\xi\subset T$) satisfying $\sigma(T,T')=\frac{l_{T'}(\xi)}{l_T(\xi)}$ a \emph{witness} for the minimal stretching factor from $T$ to $T'$. A hyperbolic group element $g\in G$ such that some (and hence any) fundamental domain for the action of $g$ on its axis $A_g\subset T$ meets each $G$-orbit of vertices of $T$ at most 10 times will be called a \emph{candidate of $T$}. Theorem~\ref{thm: inf = sup} asserts that there always exists a witness which is a candidate (our notion of candidates is nonstandard, as remarked below). We will denote by $\mathrm{cand}(T)\subset G$ the set of candidates of $T$.

If we choose for each $g\in\mathrm{cand}(T)$ a fundamental domain for the action of $g$ on its axis $A_g\subset T$, these fundamental domains project to only finitely many different edge loops in the quotient graph $G\backslash T$. In particular, the set of translation lengths $\left\{l_T(g)\ |\ g\in \mathrm{cand}(T)\right\}\subset\mathbb{R}$ is finite. At the same time, for any $T'\in\mathcal{D}$ the set $\left\{l_{T'}(g)\ |\ g\in \mathrm{cand}(T)\right\}\subset\mathbb{R}$ is finite as well, since the image of $l_{T'}$ in $\mathbb{R}$ is discrete and we have $l_{T'}(g)\leq\sigma(T,T')\cdot l_T(g)$ for all $g\in G$. Clearly, if $T,T'\in\mathcal{D}$ are $G$-equivariantly homeomorphic then $\mathrm{cand}(T)=\mathrm{cand}(T')$.

\begin{rem}
With significantly more effort, one can further show that there always exists a witness whose hyperbolic axis projects to a loop in $G\backslash T$ with certain topological properties, as was done in \cite[Proposition~3.15]{FM11} for free $F_n$-trees and in \cite[Theorem~9.10]{FM13} in the special case of irreducible $G$-trees with trivial edge stabilizers. However, the weaker finiteness properties of candidates discussed above will suffice for all our applications.
\end{rem}

In the proof of Theorem~\ref{thm: inf = sup} we will use the following characterization of witnesses:

\begin{lem}
\label{lem: characterization of witnesses}
Let $\mathcal{D}$ be a deformation space of $G$-trees and $T,T'\in\mathcal{D}$. For an optimal map $f\colon T\to T'$ and a hyperbolic group element $g\in G$, the following are equivalent:
\begin{enumerate}
\item $\sigma(f)=\frac{l_{T'}(g)}{l_T(g)}$;
\item the hyperbolic axis $A_g\subseteq T$ is contained in the tension forest $\Delta(f)$ and it is legal with respect to the train track structure defined by $f$;
\item the hyperbolic axis $A_g\subseteq T$ is contained in the tension forest $\Delta(f)$ and $f(A_g)\subseteq T'$ equals $A_g'$, the hyperbolic axis of $g$ in $T'$.
\end{enumerate}
\end{lem}

\begin{proof}

$(2)\Rightarrow (3)\Rightarrow (1)$ Since $f$ is $G$-equivariant and $A_g$ is legal, the image $f(A_g)$ is a $g$-invariant line and thus equals $A_g'$. Consequently, and since $A_g$ is assumed to lie in the tension forest $\Delta(f)$, for $p\in A_g$ we have $$l_{T'}(g) = d(g f(p),f(p)) = d(f(g p),f(p))= \sigma(f)\cdot d(g p,p) = \sigma(f)\cdot l_T(g)$$ and we conclude that $\sigma(f)=\frac{l_{T'}(g)}{l_T(g)}$.

$(1)\Rightarrow (2)$ We will argue by contradiction. First, suppose that $A_g$ is not contained in the tension forest $\Delta(f)$. Then for any $p\in A_g$ the segment $[p,gp]$ is stretched by strictly less than $\sigma(f)$ and we have
\begin{align*}
l_{T'}(g) \leq d(g f(p),f(p)) &= d(f(g p),f(p))\\
&< \sigma(f)\cdot d(g p,p) = \sigma(f)\cdot l_T(g)
\end{align*}
whence $\sigma(f)>\frac{l_{T'}(g)}{l_T(g)}$. On the other hand, if $A_g$ is contained in $\Delta(f)$ but not legal with respect to the train track structure defined by $f$ then there exists a vertex $v\in V(A_g)$ at which the two directions of $A_g$ define the same gate. The images of $[g^{-1}v,v]$ and $[v,gv]$ under $f$ then overlap in a segment of positive length and $l_{T'}(g)$ is strictly smaller than $\sigma(f)\cdot l_T(g)$, whence $\frac{l_{T'}(g)}{l_T(g)}<\sigma(f)$.
\end{proof}

\begin{proof}[Proof of Theorem~\ref{thm: inf = sup}]
Since $T$ and $T'$ are irreducible, there exists an optimal map $f\colon T\to T'$ (this is the only step in the proof that uses irreducibility). By Lemma~\ref{lem: characterization of witnesses}, in order to prove the claim, it suffices to find a hyperbolic group element $\xi\in G$ whose axis $A_\xi\subseteq T$ is contained in $\Delta=\Delta(f)$ and legal with respect to the train track structure defined by $f$. It will be clear from our construction of $\xi$ that a fundamental domain for the action of $\xi$ on $A_\xi$ meets each $G$-orbit of vertices of $T$ at most 10 times, i.e., $\xi$ is a candidate.

Since $\Delta$ has at least 2 gates at every vertex, we can find a legal ray $R\subset\Delta$ based at some vertex $v_0\in V(\Delta)$. There always exists a vertex $x\in V(R)$ such that $x=gx_0$ for some $x_0\in [v_0,x)$ and some hyperbolic group element $g\in G$, which can be seen as follows: Since $T$ is minimal and therefore cocompact, there are only finitely many $G$-orbits of vertices in $T$. We can thus find pairwise distinct vertices $x_0,x_1,x_2\in V(R)$ and $g_1,g_2\in G$ such that $x_1=g_1x_0$ and $x_2=g_2x_1$. If either $g_1$ or $g_2$ is hyperbolic, we are done. If both are elliptic, each $g_i$ fixes only the midpoint of the segment $[x_{i-1},x_i]$ and the product $g=g_2g_1$ maps $x_0$ to $x_2$. The fixed point sets of $g_1$ and $g_2$ being disjoint, $g$ is hyperbolic by Proposition~\ref{prop: translation formulas}(\ref{it: elliptics}).

We choose $x$ to be the first vertex of $R$ with this property, for which the segment $[x_0,x]$ meets the $G$-orbit of $x_0$ at most 3 times (there could lie an elliptic translate of $x_0$ in between $x_0$ and $gx_0$) and each $G$-orbit of vertices other than that of $x_0$ at most 2 times. The segment $[x_0,x]\subset R$ is then a closed fundamental domain for the action of $g$ on $A_g$ and the stretching factor of $f$ on any subsegment of $A_g$ equals that of $f$ on $[x_0,x]\subset\Delta$, whence $A_g\subseteq\Delta$. If $A_g$ is legal, we are done. If not, since all turns of $A_g$ in between $x_0$ and $x$ are legal but $A_g$ is assumed illegal, the turns at $x_0$ and $x$ must be illegal. We then have $A_g\cap R=[x_0,x]$ and we continue moving along the legal ray $R$ until we reach the first vertex $y\in V(R)$ with $y=hy_0$ for some $y_0\in(x,y)$ and some hyperbolic group element $h\in G$. Analogously, the segment $[y_0,y]$ meets the $G$-orbit of $y_0$ at most 3 times and each $G$-orbit of vertices other than that of $y_0$ at most 2 times. Note that the open segment $(x,y_0)$ meets each $G$-orbit of vertices of $T$ at most 2 times.

If $A_h\subseteq\Delta$ is legal, we are done. If not, we have $A_g\cap A_h=\emptyset$ and the product $hg$ is hyperbolic by Proposition~\ref{prop: translation formulas}(\ref{it: hyperbolics}). A closed fundamental domain for the action of $hg$ on its axis $A_{hg}$ is given by $[x_0,hx]=[x_0,x]\cup [x,y_0]\cup [y_0,y]\cup h[y_0,x]\subset\Delta$, since we have $hg[x_0,hx]\cap [x_0,hx]=\left\{hx\right\}$ (see Figure~\ref{fig: legal disjoint}). We conclude that $A_{hg}\subseteq\Delta$. In particular, the fundamental domain $[x_0,hx)$ meets each $G$-orbit of vertices of $T$ at most $3+2+3+2=10$ times and $hg$ is a candidate of $T$.
\begin{figure}[h!]
$$\begin{tikzpicture}
\node [right] at (12,0) {$A_g$};
\node [above] at (3,0) {$x_0$};
\node [above] at (6,0) {$x=gx_0$};
\node [above] at (9,0) {$gx$};
\node [right] at (4,-2) {$g^{-1}A_h$};
\node [right] at (12,-4) {$A_h$};
\node [below] at (3,-4) {$h^{-1}y_0$};
\node [below] at (6,-4) {$y_0$};
\node [above] at (9,-2) {$hx$};
\node [above] at (11,-2) {$hgx$};
\node [below] at (9,-4) {$y=hy_0$};
\node [right] at (12,-2) {$hA_g$};
\draw [-] [line width=0.075cm](3,0) -- (6,0);
\draw [-] [line width=0.075cm](6,0) -- (6,-4);
\draw [-] [line width=0.075cm](6,-4) -- (9,-4);
\draw [-] [dashed, line width=0.075cm](9,-4) -- (9,-2);
\draw [->] [line width=0.075cm](9,-2) -- (9,-2.5);
\node [left] at (9,-2.4) {$\delta_{hx,y}$};
\draw [->] [line width=0.075cm](9,-2) -- (9.5,-2);
\node [below] at (10,-2) {$hg\delta_{x_0,x}$};
\draw [->] (0,0) -- (12,0);
\draw [->] (0,-2) -- (4,-2);
\draw [-] (3,0) -- (3,-2);
\draw [-] (6,0) -- (6,-4);
\draw [-] (9,-2) -- (9,-4);
\draw [->] (8,-2) -- (12,-2);
\draw [->] (0,-4) -- (12,-4);
\draw [fill] (3,0) circle [radius=.05];
\draw [fill] (6,0) circle [radius=.05];
\draw [fill] (9,0) circle [radius=.05];
\draw [fill] (3,-4) circle [radius=.05];
\draw [fill] (6,-4) circle [radius=.05];
\draw [fill] (9,-2) circle [radius=.05];
\draw [fill] (11,-2) circle [radius=.05];
\draw [fill] (9,-4) circle [radius=.05];
\end{tikzpicture}$$
\caption{The segment $[x_0,hx]$ (both bold and dashed, where we know that the bold part lies in $R$) and the directions $hg\delta_{x_0,x}$ and $\delta_{hx,y}$ (arrows).}
\label{fig: legal disjoint}
\end{figure}
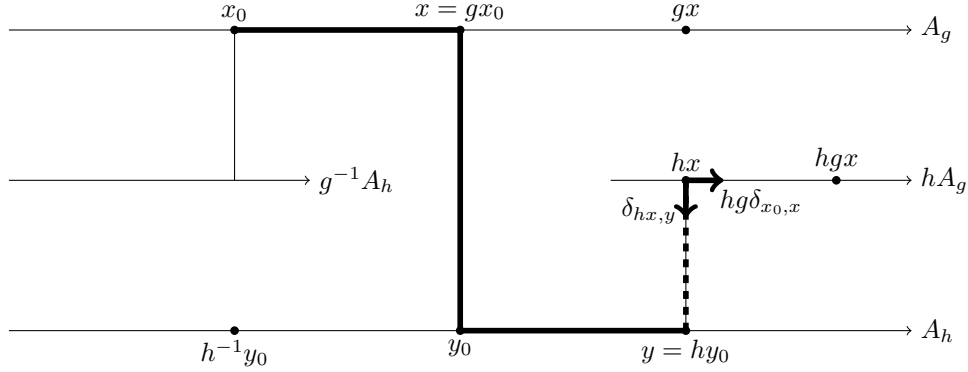

In order to show that $A_{hg}$ is legal, it suffices to show that $[x_0,hx]$ does not make any illegal turns and that the directions $hg\delta_{x_0,x}$ and $\delta_{hx,y}$ are not equivalent. By the legality of $R$, it is clear that all turns of the subsegment $[x_0,y]$ are legal. The turn of $[x_0,hx]$ at $y$ is legal if and only if $\delta_{y,y_0}\nsim \delta_{y,hx}$. This is equivalent to $\delta_{y_0,h^{-1}y_0}\nsim \delta_{y_0,x}$ but which is true since $A_h$ is assumed illegal (i.e., $\delta_{y_0,h^{-1}y_0}\sim\delta_{y_0,y}$) and $\delta_{y_0,y}\nsim\delta_{y_0,x}$ by the legality of $R$. Lastly, we need to show that $hg\delta_{x_0,x}\nsim\delta_{hx,y}$. We analogously observe that this is the case if and only if $\delta_{x,gx}\nsim\delta_{x,y_0}$ but which is true as $A_g$ is illegal (i.e., $\delta_{x,gx}\sim\delta_{x,x_0}$) and $\delta_{x,x_0}\nsim\delta_{x,y_0}$ by the legality of $R$.
\end{proof}

We give a proof of the following well-known fact in the language of trees:

\begin{prop}
\label{prop: asymmetric metric}
The Lipschitz metric on Outer space $\mathcal{PX}_n$ (Example~\ref{ex: Outer space}) is an asymmetric metric. That is, if two $F_n$-trees $T,T'\in\mathcal{PX}_n$ satisfy $d_{Lip}(T,T')=0$ then they are $F_n$-equivariantly isometric.
\end{prop}

\begin{proof}
Let $f\colon T\to T'$ be an optimal map with $\sigma(f)=1$. For $e\in E(T)$ we denote by $\sigma_e(f)$ the slope of $f$ on $e$. Since $f$ is surjective, the induced map on metric quotient graphs $F_n\backslash f\colon F_n\backslash T\to F_n\backslash T'$ is surjective as well and we have $$1=\mathrm{vol}(\mathrm{im}(F_n\backslash f))=\left(\sum_{e\in E(F_n\backslash T)}\sigma_e(f)\cdot\mathrm{length}(e)\right)-C$$ where $C\geq0$ measures overlaps of images of edges. Since $T$ has covolume 1 and $\sigma_e(f)\leq\sigma(f)=1$ for all $e\in E(T)$, we conclude that $1\leq\sigma(f)-C=1-C$, whence $C=0$. Consequently, we have $\sigma_e(f)=1$ for all $e\in E(T)$ and hence $\Delta(f)=T$.

The $F_n$-trees in $\mathcal{PX}_n$ are irreducible, and in order to prove the claim it suffices to show that for all hyperbolic (here, nontrivial) group elements $g\in F_n$ we have $l_T(g)=l_{T'}(g)$. On the one hand, if $g\in F_n$ is hyperbolic and $p\in A_g\subset T$ a point in its hyperbolic axis, we have $$l_{T'}(g)\leq d(f(p),gf(p))=d(f(p),f(gp))\leq\sigma(f)\cdot d(p,gp)=l_T(g).$$ On the other hand, suppose that there exists a hyperbolic group element $g\in F_n$ such that $l_{T'}(g)$ is strictly smaller than $l_T(g)$. Since the tension forest of $f$ is all of $T$, by Lemma~\ref{lem: characterization of witnesses} the hyperbolic axis $A_g\subset T$ cannot be legal with respect to the train track structure defined by $f$. Hence, we can find a vertex $v\in V(A_g)$ at which the turn defined by $A_g$ is not legal, i.e., at which the germs of two adjacent edges are mapped to the same germ under $f$. Since $F_n$ acts on $T$ freely, the two germs are not $F_n$-equivalent and we can find a fundamental domain $X\subset T$ for the action of $F_n$ on $T$ that contains the two germs and has volume 1. Its image $f(X)\subset T'$ is a fundamental domain for the action of $F_n$ on $T'$ whose volume is strictly smaller than 1, contradicting the fact that $T'$ has covolume 1.
\end{proof}

\begin{rem}
The proof of Proposition~\ref{prop: asymmetric metric} is specific for free $F_n$-trees, as the two germs may otherwise be $G$-equivalent (their common vertex may be stabilized by a nontrivial group element that swaps the two adjacent edges). In that case, we can no longer find a fundamental domain of volume 1 that contains both germs.
\end{rem}

\subsection{Convergent sequences}
\label{sec: convergent sequences}

In this section we relate topological convergence in projectivized deformation spaces of $G$-trees with convergence with respect to the (symmetrized) Lipschitz metric.

\begin{prop}
\label{prop: convergent sequence}
Let $\mathcal{PD}$ be a projectivized deformation space of irreducible $G$-trees and $(T_k)_{k\in\mathbb{N}}$ a sequence of $G$-trees in $\mathcal{PD}$ that converges to $T\in\mathcal{PD}$ in the weak topology. Then $\lim_{k\to\infty} d_{Lip}^{sym}(T_k,T)=0$.
\end{prop}

In the weak topology, $\mathcal{PD}$ is homeomorphic to the covolume-1-section in the unprojectivized deformation space $\mathcal{D}$. Thus, the sequence $(T_k)_{k\in\mathbb{N}}$ weakly converges to $T$ also as covolume-1-representatives in $\mathcal{D}$. The weak topology being the finest of the three topologies, $(T_k)_{k\in\mathbb{N}}$ converges to $T$ in all three topologies, where convergence in the unprojectivized axes topology means that for all $g\in G$ we have $\lim_{k\to\infty}{l_{T_k}(g)}=l_T(g)$ (pointwise convergence of translation length functions).

\begin{proof}
We will first show that $\lim_{k\to\infty}d_{Lip}(T_k,T)=0$. Let $(f_k\colon T_k\to T)_{k\in\mathbb{N}}$ be a sequence of optimal maps. By Theorem~\ref{thm: inf = sup}, for all $k\in\mathbb{N}$ there exists a candidate $\xi_k\in \mathrm{cand}(T_k)\subset G$ such that $$d_{Lip}(T_k,T)=\log\left(\frac{l_T(\xi_k)}{l_{T_k}(\xi_k)}\right).$$
Since the sequence $(T_k)_{k\in\mathbb{N}}$ converges weakly, it meets only finitely many open simplices of $\mathcal{PD}$ and the $G$-trees $(T_k)_{k\in\mathbb{N}}$ are of only finitely many $G$-equivariant homeomorphism types. After decomposing the sequence into subsequences (for each of which we will obtain the same result), we may assume that the $G$-trees are in fact all $G$-equivariantly homeomorphic, or even equal as nonmetric $G$-trees. The set of candidates $\mathrm{cand}(T_k)\subset G$ is then independent of $k$ and $(l_T(\xi_k))_{k\in\mathbb{N}}$ takes only finitely many values. After decomposing $(T_k)_{k\in\mathbb{N}}$ into subsequences once more, we may assume that $(l_T(\xi_k))_{k\in\mathbb{N}}$ is constant, say $l_T(\xi_k)=C$ for all $k\in\mathbb{N}$.

By the remarks made above, the sequence $(T_k)_{k\in\mathbb{N}}$ converges also as covolume-1-representatives in the unprojectivized axes topology. Thus, for all $K\in\mathbb{N}$ we have $\lim_{k\to\infty}{l_{T_k}(\xi_K)}=l_T(\xi_K)=C$. Recall that the candidates $(\xi_K)_{K\in\mathbb{N}}\subset G$ give rise to only finitely many different edge loops in the quotient graph $G\backslash T_1$. In fact, if two candidates $\xi_{K_1}$ and $\xi_{K_2}$ give rise to the same edge loop in $G\backslash T_1$, they give rise to the same edge loop in $G\backslash T_k$ for all $k\in\mathbb{N}$ (because the $G$-trees $(T_k)_{k\in\mathbb{N}}$ are equal as nonmetric $G$-trees). Thus, the family of sequences $\left\{(l_{T_k}(\xi_K))_{k\in\mathbb{N}}\ |\ K\in\mathbb{N}\right\}$ is finite and for all $\varepsilon>0$ there exists $N>0$ such that for all $K\in\mathbb{N}$ we have $$|C-l_{T_k}(\xi_K)|<\varepsilon$$ whenever $k\geq N$. In particular, we have $|C-l_{T_k}(\xi_k)|<\varepsilon$ whenever $k\geq N$ and we conclude that $\lim_{k\to\infty}l_{T_k}(\xi_k)=C$. Consequently, $$\lim_{k\to\infty} d_{Lip}(T_k,T)=\log\left(\lim_{k\to\infty}\frac{l_T(\xi_k)}{l_{T_k}(\xi_k)}\right)=\log\left(\frac{C}{\lim_{k\to\infty}l_{T_k}(\xi_k)}\right)=\log\left(1\right)=0.$$
Showing that $\lim_{k\to\infty} d_{Lip}(T,T_k)=0$ is similar but easier, because it does not require that the sequence $(T_k)_{k\in\mathbb{N}}$ meets only finitely many open simplices of $\mathcal{PD}$.
\end{proof}

As for the converse of Proposition~\ref{prop: convergent sequence}, we have the following:

\begin{prop}
\label{prop: symmetrized convergent sequence}
Let $\mathcal{PD}$ be a projectivized deformation space of irreducible $G$-trees and $(T_k)_{k\in\mathbb{N}}$ a sequence of $G$-trees in $\mathcal{PD}$ such that for some $T\in\mathcal{PD}$ we have $\lim_{k\to\infty} d_{Lip}^{sym}(T_k,T)=0$. Then $(T_k)_{k\in\mathbb{N}}$ converges to $T$ in the axes topology.
\end{prop}

In contrast to convergence in the unprojectivized axes topology, the sequence $(T_k)_{k\in\mathbb{N}}$ converges to $T$ in the projectivized axes topology if there exist positive real numbers $(C_k)_{k\in\mathbb{N}}$ such that for all $g\in G$ we have $\lim_{k\to\infty}{C_k\cdot l_{T_k}(g)}=l_T(g)$ (pointwise convergence of projectivized translation length functions).

\begin{proof}
We will argue as in the proof of \cite[Theorem~4.11]{FM11}. For any positive real-valued function $f$ satisfying $\sup\frac{1}{f(x)}<\infty$ we have $\sup\frac{1}{f(x)}=\frac{1}{\inf f(x)}$. Therefore, since $$\frac{1}{\frac{l_{T_k}(g)}{l_T(g)}}=\frac{l_T(g)}{l_{T_k}(g)}\leq\sigma(T_k,T)<\infty$$ for all hyperbolic group elements $g\in G$, we have $$\lim_{k\to\infty} d_{Lip}^{sym}(T_k,T)=0\quad\Leftrightarrow\quad\lim_{k\to\infty}\frac{\sup_{g}\frac{l_{T_k}(g)}{l_T(g)}}{\inf_{g}\frac{l_{T_k}(g)}{l_T(g)}}=1.$$ Assuming that $\lim_{k\to\infty} d_{Lip}^{sym}(T_k,T)=0$, we conclude that for all $\varepsilon>0$ there exists $K\in\mathbb{N}$ such that for all $k\geq K$ we have
\begin{equation}
\inf_{g}\frac{l_{T_k}(g)}{l_T(g)} \leq \sup_{g}\frac{l_{T_k}(g)}{l_T(g)}\leq \inf_{g}\frac{l_{T_k}(g)}{l_T(g)}\cdot (1+\varepsilon).
\label{eq: inf and sup inequality}
\end{equation}

Clearly, for all hyperbolic group elements $\xi\in G$ we have $$\inf_{g}\frac{l_{T_k}(g)}{l_T(g)}\leq\frac{l_{T_k}(\xi)}{l_T(\xi)}\leq\sup_{g}\frac{l_{T_k}(g)}{l_T(g)}.$$ Setting $I_k=\inf_{g}\frac{l_{T_k}(g)}{l_T(g)}$, inequality (\ref{eq: inf and sup inequality}) implies that $I_k\leq\frac{l_{T_k}(\xi)}{l_T(\xi)}\leq I_k\cdot (1+\varepsilon)$ whenever $k\geq K$. In particular, the unprojectivized translation length functions $(\frac{1}{I_k}l_{T_k})_{k\in\mathbb{N}}$ converge to $l_T$ uniformly and \emph{a fortiori} pointwise. We conclude that the $G$-trees $(T_k)_{k\in\mathbb{N}}$ converge to $T$ in the projectivized axes topology.
\end{proof}

Recall from Section~\ref{sec: topologies} that if $\mathcal{PD}$ is a projectivized deformation space of locally finite irreducible $G$-trees with finitely generated vertex stabilizers then the equivariant Gromov-Hausdorff topology, the axes topology, and the weak topology agree on $\mathcal{PD}$.

\begin{cor}
\label{cor: symmetrized metric topology}
If $\mathcal{PD}$ consists of locally finite irreducible $G$-trees with finitely generated vertex stabilizers then the symmetrized Lipschitz metric $d_{Lip}^{sym}$ induces the standard topology on $\mathcal{PD}$.
\end{cor}

\begin{proof}
Since the locally finite complex $\mathcal{PD}$ is metrizable, it suffices to show that the two topologies have the same convergent sequences. This immediately follows from Propositions~\ref{prop: convergent sequence} and \ref{prop: symmetrized convergent sequence} and the fact that the three topologies agree on $\mathcal{PD}$.
\end{proof}

\begin{ex}
The symmetrized Lipschitz metric induces the standard topology on the projectivized deformation spaces discussed in Examples~\ref{ex: Outer space}, \ref{ex: virtually free group}, and \ref{ex: nonelementary GBS groups}.
\end{ex}

\subsection{Folding paths and geodesics}
\label{sec: geodesics}

Let $\mathcal{D}$ be a deformation space of $G$-trees and $T,T'\in\mathcal{D}$. A $G$-equivariant map $f\colon T\to T'$ is \emph{simplicial} if it maps each edge of $T$ isometrically to an edge of $T'$. A $G$-equivariant map $f\colon T\to T'$ is a \emph{morphism} if it is an isometry on edges (but not necessarily simplicial) or, equivalently, if the simplicial structures on $T$ and $T'$ may be subdivided such that $f$ becomes simplicial.

Suppose we are given two $G$-trees $T,T'\in\mathcal{D}$ and a morphism $f\colon T\to T'$. Skora \cite{Sk89} has described a technique of ``folding $T$ along $f$'' to obtain a 1-parameter family of $G$-trees $(T_t)_{t\in[0,\infty]}$ together with morphisms $\phi_t\colon T\to T_t$ and $\psi_t\colon T_t\to T'$ such that
\begin{itemize}
\item $T_0=T$ and $T_\infty=T'$;
\item $\phi_0=id_T$, $\phi_\infty=\psi_0=f$, and $\psi_\infty=id_{T'}$;
\item for all $t\in [0,\infty]$ the following diagram commutes:
$$\begin{xy}\xymatrix{T\ar[rr]^f\ar[dr]_{\phi_t} & & T'\\ & T_t\ar[ru]_{\psi_t} &}\end{xy}$$
\end{itemize}
Explicitly, for $t\in[0,\infty]$ we let $\sim_t$ be the equivalence relation on $T$ generated by $x\sim_t y$ if $f(x)=f(y)$ and $f([x,y])\subseteq D_t(f(x))$, where $D_t(f(x))$ denotes the closed ball of radius $t$ around $f(x)\in T'$. As a set, we define $T_t$ as the quotient $T/\sim_t$. We let $\phi_t\colon T\to T_t$ be the $G$-equivariant quotient map and $\psi_t\colon T_t\to T'$ the unique induced $G$-equivariant map, and we equip $T_t$ with the maximal metric making both $\phi_t$ and $\psi_t$ 1-Lipschitz.

For all $t\in[0,\infty]$, after possibly restricting to the unique minimal $G$-invariant subtree, the $G$-tree $T_t$ lies in the deformation space $\mathcal{D}$. The map $[0,\infty]\to\mathcal{D},\ t\to T_t$ is continuous in the equivariant Gromov-Hausdorff topology (see \cite{Sk89}, \cite{Cl05}, or \cite{GL07fp}) and therefore also in the axes topology. The intermediate $G$-trees $(T_t)_{t\in[0,\infty]}$ are contained in a finite union of open cones \cite[Lemma~6.5]{GL07}, because of which the map is also continuous in the weak topology. We obtain a path $[0,\infty]\to\mathcal{PD},\ t\to T_t$ that is continuous in all three topologies.

As in \cite{FM11} in the special case of Outer space, one can make use of folding paths to construct geodesics in projectivized deformation spaces of irreducible $G$-trees:

\begin{de}
\label{de: geodesic}
Let $\mathcal{PD}$ be a projectivized deformation space of $G$-trees. A path $\gamma\colon[a,b]\to\mathcal{PD},\ t\mapsto\gamma(t)$ with $a<b\in\mathbb{R}$ is \emph{$d_{Lip}$-continuous} if for all convergent sequences $(x_n)_{n\in\mathbb{N}}\subset [a,b]$ with $\lim_{n\to\infty}{x_n}=x$ we have $$\lim_{n\to\infty}{d_{Lip}(\gamma(x_n),\gamma(x))}=0\quad\text{and}\quad\lim_{n\to\infty}{d_{Lip}(\gamma(x),\gamma(x_n))}=0.$$
We say that a $d_{Lip}$-continuous path $\gamma\colon[a,b]\to\mathcal{PD},\ t\mapsto\gamma(t)$ with $a<b\in\mathbb{R}$ is a \emph{$d_{Lip}$-geodesic} if for all $x<y<z\in[a,b]$ we have $$d_{Lip}(\gamma(x),\gamma(y))+d_{Lip}(\gamma(y),\gamma(z))=d_{Lip}(\gamma(x),\gamma(z)).$$
\end{de}

\begin{rem}
In metric spaces, geodesics in the above sense can be reparametrized to have unit speed. However, since $d_{Lip}$ is an asymmetric pseudometric, unit speed reparametrizations in $\mathcal{PD}$ need not always exist.
\end{rem}

\begin{lem}\label{lem: geodesics}
Let $\mathcal{PD}$ be a projectivized deformation space of irreducible $G$-trees and $\gamma\colon[a,b]\to\mathcal{PD}$ be a $d_{Lip}$-continuous path with $a<b\in\mathbb{R}$. If for all $x<y<z\in[a,b]$ there exists a hyperbolic group element $\xi\in G$ such that
\begin{equation}
\label{eq: double witness}
\sigma(\gamma(x),\gamma(y))=\frac{l_{\gamma(y)}(\xi)}{l_{\gamma(x)}(\xi)}\quad\text{and}\quad\sigma(\gamma(y),\gamma(z))=\frac{l_{\gamma(z)}(\xi)}{l_{\gamma(y)}(\xi)}
\end{equation}
then $\gamma$ is a $d_{Lip}$-geodesic.
\end{lem}

\begin{proof}
We have $$\sup_g\frac{l_{\gamma(z)}(g)}{l_{\gamma(x)}(g)}\geq\frac{l_{\gamma(z)}(\xi)}{l_{\gamma(x)}(\xi)}=\frac{l_{\gamma(y)}(\xi)}{l_{\gamma(x)}(\xi)}\cdot\frac{l_{\gamma(z)}(\xi)}{l_{\gamma(y)}(\xi)}=\sup_g\left(\frac{l_{\gamma(y)}(g)}{l_{\gamma(x)}(g)}\right)\cdot\sup_g\left(\frac{l_{\gamma(z)}(g)}{l_{\gamma(y)}(g)}\right)$$ and hence $d_{Lip}(\gamma(x),\gamma(z))\geq d_{Lip}(\gamma(x),\gamma(y))+d_{Lip}(\gamma(y),\gamma(z))$, from which we conclude that $d_{Lip}(\gamma(x),\gamma(z))= d_{Lip}(\gamma(x),\gamma(y))+d_{Lip}(\gamma(y),\gamma(z))$.
\end{proof}

By Proposition~\ref{prop: convergent sequence}, if $\mathcal{PD}$ is irreducible then any path in $\mathcal{PD}$ that is continuous in the weak topology -- such as the folding path $[0,\infty]\to\mathcal{PD},\ t\to T_t$ described above -- is $d_{Lip}$-continuous.

\begin{thm}[Existence of $d_{Lip}$-geodesics]
\label{thm: geodesics}
If $\mathcal{PD}$ is a projectivized deformation space of irreducible $G$-trees then for all $T,T'\in\mathcal{PD}$ there exists a $d_{Lip}$-geodesic $\gamma\colon [0,1]\to\mathcal{PD}$ with $\gamma(0)=T$ and $\gamma(1)=T'$.
\end{thm}

\begin{proof}
Let $f\colon T\to T'$ be an optimal map and $\xi\in G$ a witness for the distance from $T$ to $T'$. By Lemma~\ref{lem: geodesics}, it suffices to construct a path $\gamma\colon[0,1]\to\mathcal{PD}$ from $T$ to $T'$ such that for all $x<y<z\in[0,1]$ we have (\ref{eq: double witness}). We will construct such a path in the unprojectivized deformation space $\mathcal{D}$, and since any witness for the minimal stretching factor between two $G$-trees remains a witness after scaling the metrics on the trees, the projection of the path to $\mathcal{PD}$ will still satisfy (\ref{eq: double witness}). In order to do so, we again regard $T$ and $T'$ as their covolume-1-representatives in $\mathcal{D}$. Let $$C=\exp(d_{Lip}(T,T'))=\frac{l_{T'}(\xi)}{l_T(\xi)}$$ and let $\overline{T}$ be the $G$-tree obtained from $T$ by $G$-equivariantly shrinking each edge of $T$ that is mapped to a point under $f$ to length 0 (collapsing these edges does not create any new elliptic subgroups, as the $G$-equivariant map $f\colon T\to T'$ factors through the quotient) and $G$-equivariantly shrinking all other edges so that they are stretched by the factor $C$ under $f$. Note that we only shrink edges in the complement of the tension forest $\Delta(f)$. Then, homothete $\overline{T}$ to $C\overline{T}$ such that $f\colon C\overline{T}\to T'$ becomes an isometry on edges, i.e., a morphism. We may now fold $C\overline{T}$ along $f$ to obtain a family of $G$-trees $(T_t)_{t\in[0,\infty]}$ that interpolate between $C\overline{T}$ and $T'$ as explained above (see Figure~\ref{fig: geodesic} for a structural sketch).
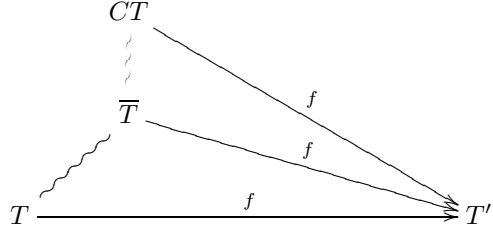
\begin{figure}[h!]
$$\begin{xy}\xymatrix{
& C\overline{T}\ar[rrrrdd]^f & & & & &\\
& \overline{T}\ar@{~}[u]\ar[rrrrd]^f & & & & &\\
T\ar@{~}[ru]\ar[rrrrr]^f & & & & & T'}\end{xy}$$
\caption{The path from $T$ to $\overline{T}$ to $C\overline{T}$ to $T'$ in $\mathcal{D}$ projects to a geodesic from $T$ to $T'$ in $\mathcal{PD}$.}
\label{fig: geodesic}
\end{figure}
This produces a path $\gamma\colon[0,1]\to\mathcal{D}$ from $T$ to $T'$ that is continuous in all three topologies and also with respect to $d_{Lip}$.

We claim that for every $G$-tree $S$ in between $T$ and $C\overline{T}$ we have $$\sigma(T,S)=\frac{l_S(\xi)}{l_T(\xi)}\quad\text{and}\quad\sigma(S,T')=\frac{l_{T'}(\xi)}{l_S(\xi)}.$$ Analogously, we claim that for every intermediate $G$-tree $T_t$ in between $C\overline{T}$ and $T'$ we have $$\sigma(T,T_t)=\frac{l_{T_t}(\xi)}{l_T(\xi)}\quad\text{and}\quad\sigma(T_t,T')=\frac{l_{T'}(\xi)}{l_{T_t}(\xi)}.$$ As for any $a\leq s\leq t\leq b$ the same construction of a path from $\gamma(s)$ to $\gamma(t)$ yields precisely the restriction of $\gamma$ to $[s,t]$, this then proves that $\gamma$ satisfies (\ref{eq: double witness}).

First, consider a $G$-tree $S$ that lies in between $T$ and $\overline{T}$. As $S$ is obtained from $T$ by shrinking edges of $T$, we have $\sigma(T,S)\leq 1$. However, as we only shrink edges outside of $\Delta(f)$, the hyperbolic axis $A_\xi\subset\Delta(f)$ is not touched and we have $\frac{l_S(\xi)}{l_T(\xi)}=1$. We may immediately deduce from this that $\sigma(T,S)=\frac{l_S(\xi)}{l_T(\xi)}$, as for all hyperbolic group elements $g\in G$ we have $\sigma(T,S)\geq\frac{l_S(g)}{l_T(g)}$ (see Lemma~\ref{lem: pre-witness}). Likewise, the map $f\colon S\to T'$ still has Lipschitz constant $C$ so that $\sigma(S,T')\leq C$. The axis $A_\xi\subset\Delta(f)\subset S$ remains legal and is stretched by the factor $C$, whence $\frac{l_{T'}(\xi)}{l_S(\xi)}=C$ and therefore $\sigma(S,T')=\frac{l_{T'}(\xi)}{l_S(\xi)}$. 

Analogously, if $S$ lies in between $\overline{T}$ and $C\overline{T}$, say $S=C'\overline{T}$ with $C'\in [1,C]$, then $\sigma(T,S)\leq C'$ and $\frac{l_S(\xi)}{l_T(\xi)}=C'$, whence $\sigma(T,S)=\frac{l_S(\xi)}{l_T(\xi)}$. The map $f\colon S\to T'$ has Lipschitz constant $\frac{C}{C'}$ and the axis $A_\xi\subset\Delta(f)\subset S$ is stretched by $\frac{C}{C'}$. We conclude that $\sigma(S,T')=\frac{l_{T'}(\xi)}{l_S(\xi)}$.

Consider now an intermediate $G$-tree $T_t$ in between $C\overline{T}$ and $T'$. As the quotient map $\phi_t\colon C\overline{T}\to T_t$ is 1-Lipschitz, the composition $T\stackrel{id}{\to} C\overline{T}\stackrel{\phi_t}{\to} T_t$ is $C$-Lipschitz. The hyperbolic axis $A_\xi\subset\Delta(f)\subset T$ is legal with respect to $f$ and hence does not get folded in $T_t=C\overline{T}/\sim_t$. We therefore have $\frac{l_{T_t}(\xi)}{l_T(\xi)}=C$, whence $\sigma(T,T_t)=\frac{l_{T_t}(\xi)}{l_T(\xi)}$. Analogously, the induced map $\psi_t\colon T_t\to T'$ is 1-Lipschitz and the hyperbolic axis $A_\xi\subset\Delta(\psi_t)\subset T_t$ is legal with respect to $\psi_t$. We conclude that $\frac{l_{T'}(\xi)}{l_{T_t}(\xi)}=1$ and hence that $\sigma(T_t,T')=\frac{l_{T'}(\xi)}{l_{T_t}(\xi)}$.
\end{proof}

\section{Displacement functions}
\label{sec: classification}

Let $\mathcal{PD}$ be a projectivized deformation space of $G$-trees and $\Phi\in \Out_\mathcal{D}(G)$. We equip $\mathcal{PD}$ with the Lipschitz metric $d_{Lip}$ and define the \emph{displacement function} associated to $\Phi$ as the function $$\widetilde{\Phi}\colon\mathcal{PD}\to\mathbb{R}_{\geq 0},\ T\mapsto d_{Lip}(T,T\Phi).$$ We call $\Phi$ \emph{elliptic} if $\inf \widetilde{\Phi}=0$ and the infimum is realized. We say that $\Phi$ is \emph{hyperbolic} if $\inf \widetilde{\Phi}>0$ and the infimum is realized. Lastly, we say that $\Phi$ is \emph{parabolic} if $\inf \widetilde{\Phi}$ is not realized.

\subsection{Elliptic automorphisms}
\label{sec: elliptics}

If $\Phi\in \Out_\mathcal{D}(G)$ is elliptic then, by definition, there exists a $G$-tree $T\in\mathcal{PD}$ such that $d_{Lip}(T,T\Phi)=0$. One would like to conclude that $T$ lies in the fixed point set of $\Phi$, but from Example~\ref{ex: not an asymmetric metric} we know that the asymmetric pseudometric $d_{Lip}$ fails to be an asymmetric metric, i.e., $d_{Lip}(T,T')=0$ does generally not imply that $T$ and $T'$ are $G$-equivariantly isometric. However, the $G$-trees $T$ and $T'$ in the counterexample are not homeomorphic and thus they do not lie in the same $\Out_{\mathcal{D}}(G)$-orbit, for they would otherwise have the same underlying metric simplicial tree. Therefore, one may still ask whether $d_{Lip}$ is an asymmetric metric on $\Out_D(G)$-orbits. As we will see, the general answer is ``no" (Example~\ref{ex: counter elliptic}) but it is ``yes" in certain cases (Proposition~\ref{prop: separation no nontrivial moduli}). This answers a question in an earlier preprint of this paper. The arguments in this section arose out of discussions with Camille Horbez and Gilbert Levitt.

\subsubsection*{The separation property of $d_{Lip}$ on $\Out_D(G)$-orbits}

Let $T\in\mathcal{PD}$ and $\Phi\in \Out_\mathcal{D}(G)$ such that $d_{Lip}(T,T\Phi)=0$. If $T$ is irreducible then there exists an optimal map $f\colon T\to T\Phi$ with $\sigma(f)=1$, and one easily shows (as in the proof of Proposition~\ref{prop: asymmetric metric}) that $f$ has stretching factor 1 on all edges of $T$. After subdividing the simplicial structures on $T$ and $T\Phi$ (independently of each other) by $G$-equivariantly adding redundant vertices, $f$ becomes simplicial (as defined in Section~\ref{sec: geodesics}). We will denote the subdivided $G$-trees again by $T$ and $T\Phi$.

If all edge stabilizers of $T$ are finitely generated then by \cite[Section~2]{BF91} the simplicial map $f$ factors as a finite composition of $G$-equivariant simplicial quotient maps, so-called \emph{folds}, which can be classified into types IA-IIIA, IB-IIIB, and IIIC (we refer the reader to \cite{BF91} for definitions). All folds other than type IIA and IIB folds irreversibly decrease the metric covolume, so they cannot occur. After subdividing the simplicial structure on $T$ once more, a type IIB fold is a composition of two type IIA folds (these subdivisions add only a finite number of $G$-orbits of vertices), so we may assume that $f$ factors as a finite composition of type IIA folds. Explicitly, a type IIA fold is a simplicial quotient map $T\to T/\sim$, where $\sim$ is a $G$-equivariant equivalence relation on $T$ that is of the following form: There are distinct edges $e_1,e_2\in E(T)$ with $\iota(e_1)=\iota(e_2)\in V(T)$ and a group element $g\in G_{\iota(e_1)}$ such that $ge_1=e_2$, and $\sim$ is the equivalence relation generated by $h e_1\sim h e_2$ for all $h\in G$. Intuitively, on the level of quotient graphs of groups, performing a type IIA fold corresponds to pulling an element of a vertex stabilizer along an edge (see Figure~\ref{fig: type IIA fold}).
\begin{figure}[h!]
$$\begin{tikzpicture}
\node [left] at (1,0) {$G$};
\node [below] at (2,0) {$E$};
\node [right] at (3,0) {$H$};
\node [left] at (8,0) {$G$};
\node [below] at (9,0) {$\langle E,g\rangle$};
\node [right] at (10,0) {$\langle H,g\rangle$};

\draw [-] (1,0) -- (3,0);
\draw [->] (4.25,0) -- (6.75,0);
\node [above] at (5.5,0) {type IIA fold};
\node [below] at (5.5,0) {$g\in G,\ g\notin E$};
\draw [-] (8,0) -- (10,0);

\draw [fill] (1,0) circle [radius=.05];
\draw [fill] (3,0) circle [radius=.05];
\draw [fill] (8,0) circle [radius=.05];
\draw [fill] (10,0) circle [radius=.05];
\end{tikzpicture}$$
\caption{The effect of a type IIA fold on the quotient graph of groups.}
\label{fig: type IIA fold}
\end{figure}

A type IIA fold always enlarges but never reduces an edge group. We will make use of this behavior to confirm the separation property of $d_{Lip}$ on $\Out_\mathcal{D}(G)$-orbits in the following special case:

\begin{prop}[Levitt]
\label{prop: separation no nontrivial moduli}
Let $\mathcal{PD}$ be a projectivized deformation space of locally finite irreducible $G$-trees with finitely generated edge stabilizers. If $\mathcal{PD}$ has no nontrivial integral modulus (see Section~\ref{sec: automorphisms acting}) and if $T\in\mathcal{PD}$ and $\Phi\in \Out_\mathcal{D}(G)$ satisfy $d_{Lip}(T,T\Phi)=0$ then $T$ and $T\Phi$ are $G$-equivariantly isometric.
\end{prop}

Before we turn to the proof of Proposition~\ref{prop: separation no nontrivial moduli}, we discuss the existence of \emph{maximal elliptic subgroups}, i.e., elliptic subgroups that are not properly contained in any other elliptic subgroup. A maximal elliptic subgroup is always a vertex stabilizer.

\begin{lem}
\label{lem: maximal elliptic}
Let $\mathcal{PD}$ be a projectivized deformation space of locally finite $G$-trees. If $\mathcal{PD}$ has no nontrivial integral modulus then for any $G$-tree $T\in\mathcal{PD}$ and any edge $e\in E(T)$ the edge group $G_e$ is contained in a maximal elliptic subgroup of $T$.
\end{lem}

\begin{proof}
We first observe that, under these assumptions, for any vertex $v\in V(T)$ the vertex group $G_v\leq G$ is not properly contained in a conjugate of itself: Suppose to the contrary that there exists a vertex $v\in V(T)$ such that $G_v$ is a proper subgroup of $gG_vg^{-1}$ for some $g\in G$. We then have $$\mu(g)=\frac{[G_v:(G_v\cap gG_vg^{-1})]}{[gG_vg^{-1}:(G_v\cap gG_vg^{-1})]}=\frac{1}{[gG_vg^{-1}:G_v]}$$ with $[gG_vg^{-1}:G_v]>1$, in which case $\mu(g^{-1})=\frac{1}{\mu(g)}$ is a nontrivial integral modulus, contradicting our assumptions. To prove the lemma, we again argue by contradiction: Suppose that the edge group $G_e$ is not contained in a maximal elliptic subgroup. Each vertex group adjacent to $e$ is then properly contained in another vertex group, which is again properly contained in yet another vertex group. Inductively, we obtain an infinite properly ascending chain of vertex groups that lie in only finitely many conjugacy classes by the cocompactness of $T$. We conclude that there exists a vertex $v\in V(T)$ and a group element $g\in G$ such that $G_v$ is a proper subgroup of $gG_vg^{-1}$, which contradicts the first part of the proof.
\end{proof}

\begin{proof}[Proof of Proposition~\ref{prop: separation no nontrivial moduli}]
Since $T$ is irreducible and has finitely generated edge stabilizers, after subdividing the simplicial structures on $T$ and $T\Phi$ there exists a $G$-equivariant simplicial map $f\colon T\to T\Phi$ that factors as a finite composition of type IIA folds. We claim that $T\Phi$ cannot be obtained from $T$ by nontrivial type IIA folds, whence $T$ and $T\Phi$ are $G$-equivariantly isometric:

By Lemma~\ref{lem: maximal elliptic}, the stabilizer $G_e$ of any edge $e\in E(T)$ is contained in a maximal elliptic subgroup of $T$, which is always a vertex stabilizer. Let $M_i\leq G,\ i\in I$ be the maximal elliptic subgroups of $T$ that contain $G_e$. Since $T$ has only finitely many $G$-orbits of vertices, the vertex groups $M_i,\ i\in I$ fall into only finitely many conjugacy classes, and we assume for a moment that they are in fact all conjugate. Then, for a distinguished maximal elliptic subgroup $M$ containing $G_e$, the image of the modular homomorphism $\mu\colon G\to (\mathbb{Q}_{>0},\times)$ defined by $$\mu(g)=\frac{[M:(M\cap gMg^{-1})]}{[gMg^{-1}:(M\cap gMg^{-1})]}$$ contains the values $$\frac{[M_i:(M\cap M_i)]}{[M:(M\cap M_i)]}=\frac{[M_i:(M\cap M_i)]}{[M:(M\cap M_i)]}\cdot\frac{[(M\cap M_i):G_e]}{[(M\cap M_i):G_e]}=\frac{[M_i:G_e]}{[M:G_e]},\ i\in I.$$ Since $\mathcal{PD}$ has no nontrivial integral modulus, \cite[Lemma~8.1]{Fo06} implies that the indices $[M_i:G_e],\ i\in I$ can take only finitely many values. Consequently, there exists a maximum index $\mathcal{I}(G_e)$ of $G_e$ in the maximal elliptic subgroups $M_i,\ i\in I$. If the maximal elliptic subgroups containing $G_e$ are not all conjugate, we associate to each of their finitely many conjugacy classes the maximum index of $G_e$ and define $\mathcal{I}(G_e)$ as the sum of these. One readily sees that for all $g\in G$ we have $\mathcal{I}(gG_eg^{-1})=\mathcal{I}(G_e)$ so that we have $\mathcal{I}(G_e)=\mathcal{I}(G_{e'})$ if $e$ and $e'$ lie in the same $G$-orbit of edges of $T$. Finally, let $$\mathcal{I}(T)=\sum_{e\in G\backslash T}{\mathrm{length}(e)\cdot\mathcal{I}(G_e)}$$ where $e$ ranges over the finitely many edges in the metric quotient graph of groups of $T$. The value $\mathcal{I}(T)$ is insensitive to simplicial subdivisions of $T$ and for all $\Phi\in \Out_\mathcal{D}(G)$ we have $\mathcal{I}(T\Phi)=\mathcal{I}(T)$. On the other hand, after performing a type IIA fold, for the enlarged edge group $\langle E,g\rangle$ we have $\mathcal{I}(\langle E,g\rangle)<\mathcal{I}(E)$, whereas all other edge groups are left invariant. Thus, if $T'\in\mathcal{PD}$ is obtained from $T$ by a nontrivial sequence of type IIA folds then $\mathcal{I}(T')<\mathcal{I}(T)$, whence the claim.
\end{proof}

\begin{ex}
\label{ex: separation virtually free}
Let $G$ be a finitely generated virtually nonabelian free group and $\mathcal{PD}$ the projectivized deformation space of minimal $G$-trees with finite vertex stabilizers (Example~\ref{ex: virtually free group}). We know from Example~\ref{ex: virtually free} that $\mathcal{PD}$ has no nontrivial integral modulus. Hence, if for $T\in\mathcal{PD}$ and $\Phi\in \Out_\mathcal{D}(G)=\Out(G)$ we have $d_{Lip}(T,T\Phi)=0$ then $T$ and $T\Phi$ are $G$-equivariantly isometric.
\end{ex}

\begin{cor}
Let $G$ be a finitely generated virtually nonabelian free group and $\mathcal{PD}$ the projectivized deformation space of minimal $G$-trees with finite vertex stabilizers. An automorphism $\Phi\in \Out_\mathcal{D}(G)=\Out(G)$ is elliptic with respect to $d_{Lip}$ if and only if it has finite order.
\end{cor}

\begin{proof}
It follows from \cite[Proposition 8.6]{GL07} that $\Out(G)$ acts on $\mathcal{PD}$ with finite point stabilizers. Thus, and by Example~\ref{ex: separation virtually free}, every elliptic automorphism $\Phi\in \Out(G)$ has finite order. Conversely, by \cite{Cl07} every finite-order automorphism of $G$ has a fixed point in $\mathcal{PD}$ and thus is elliptic.
\end{proof}

\begin{ex}
\label{ex: separation GBS}
Let $G$ be a nonelementary GBS group that contains no solvable Baumslag-Solitar group $\BS(1,n)$ with $n\geq 2$. Let $\mathcal{PD}$ be the projectivized deformation space of minimal $G$-trees with infinite cyclic vertex and edge stabilizers (Example~\ref{ex: nonelementary GBS groups}). By Lemma~\ref{lem: GBS integral modulus}, $\mathcal{PD}$ has no nontrivial integral modulus. Thus, if for $T\in\mathcal{PD}$ and $\Phi\in \Out_\mathcal{D}(G)=\Out(G)$ we have $d_{Lip}(T,T\Phi)=0$ then $T$ and $T\Phi$ are $G$-equivariantly isometric.
\end{ex}

\begin{rem}
For $G$ and $\mathcal{PD}$ as in Example~\ref{ex: separation GBS}, let $b$ be the first Betti number of the topological space $G\backslash T$ for any $T\in\mathcal{PD}$; this number is an invariant of $\mathcal{PD}$, as it is not affected by elementary deformations. The stabilizer of each $T\in\mathcal{PD}$ under the action of $\Out(G)$ on $\mathcal{PD}$ is virtually isomorphic to $\mathbb{Z}^k$, where $k=b$ or $b-1$ depending on $G$ \cite[Theorem 3.10]{Le07}.
\end{rem}

However, the asymmetric Lipschitz pseudometric $d_{Lip}$ does not restrict to an asymmetric metric on $\Out_\mathcal{D}(G)$-orbits in general:

\begin{ex}[Horbez]
\label{ex: counter elliptic}
Let $G=\BS(1,6)\ast F_2=\langle x,t\ |\ txt^{-1}=x^6\rangle\ast F_2$ and
consider the graph of groups decompositions $\Gamma$ and $\Gamma'$ of $G$ shown in Figure~\ref{fig: orbit counterexample}, where all edge group inclusions are the obvious ones and all edges have length $\frac{1}{3}$.
\begin{figure}[h!]
$$\begin{tikzpicture}
\draw [-] (1,.35) circle (10pt);
\draw [-] (7,.35) circle (10pt);
\draw [-] (1,-2.35) circle (10pt);
\draw [-] (7,-2.35) circle (10pt);

\node [left] at (.5,0) {$\Gamma$};
\node [left] at (6.5,0) {$\Gamma'$};

\node [above] at (1,.7) {1};
\node [right] at (1,-.2) {$\langle x,t\ |\ txt^{-1}=x^6\rangle$};
\node [right] at (1,-1) {$\langle x^3\rangle$};
\node [right] at (1,-1.8) {$\langle x^3\rangle$};
\node [below] at (1,-2.7) {1};

\node [above] at (7,.7) {1};
\node [right] at (7,-.2) {$\langle x,t\ |\ txt^{-1}=x^6\rangle$};
\node [right] at (7,-1) {$\langle x\rangle$};
\node [right] at (7,-1.8) {$\langle x\rangle$};
\node [below] at (7,-2.7) {1};

\draw [-] (1,0) -- (1,-2);
\draw [-] (7,0) -- (7,-2);

\draw [fill] (1,0) circle [radius=.05];
\draw [fill] (1,-2) circle [radius=.05];
\draw [fill] (7,-2) circle [radius=.05];
\draw [fill] (7,0) circle [radius=.05];
\end{tikzpicture}$$
\caption{The Bass-Serre trees of the graphs of groups shown above lie in the same $\Out_\mathcal{D}(G)$-orbit. They are irreducible but not locally finite.}
\label{fig: orbit counterexample}
\end{figure}
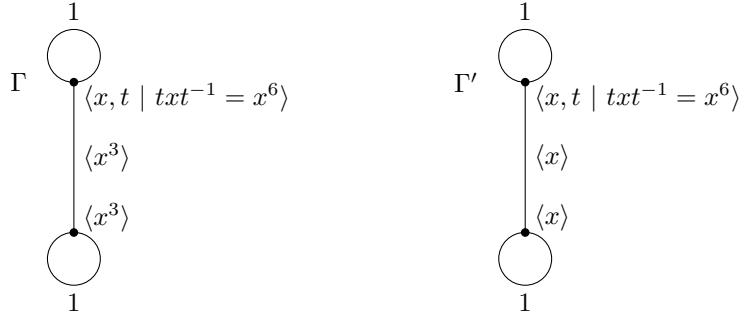
Let $T$ and $T'$ be the corresponding $G$-trees. The automorphism $$\varphi\colon \BS(1,6)\stackrel{\cong}{\to} \BS(1,6),\ x\mapsto x^3,\ t\mapsto t$$ induces an automorphism $\Phi=\varphi\ast id_{F_2}\in \Aut(G)$ and we have $T'=T\Phi$. Similarly as in Example~\ref{ex: not an asymmetric metric}, the natural morphism of graphs of groups from $\Gamma$ to $\Gamma'$ lifts to a $G$-equivariant map from $T$ to $T\Phi$ (namely, a type IIA fold) that is an isometry on edges and thus has Lipschitz constant 1, whence $d_{Lip}(T,T\Phi)=0$. However, $T$ and $T\Phi$ are not $G$-equivariantly isometric, as the group element $x\in G$ stabilizes an edge in $T\Phi$ but not in $T$ ($x$ is not a conjugate of $x^3$).
\end{ex}

\subsection{Nonparabolic automorphisms}
\label{sec: nonparabolics}

Let $\mathcal{PD}$ be a projectivized deformation space of irreducible $G$-trees and $\Phi\in \Out_\mathcal{D}(G)$ a nonparabolic automorphism, i.e., $\inf\widetilde{\Phi}$ is realized. Let $T\in\mathcal{PD}$ such that
$d_{Lip}(T,T\Phi)=\inf\widetilde{\Phi}$ and let $f\colon T\to T\Phi$ be an optimal map with tension forest $\Delta=\Delta(f)\subset T$. The following observation will be used in the proof of Theorem~\ref{thm: train tracks}:

\begin{prop}
\label{prop: nonparabolic tension forest invariant}
After a small perturbation of the metric on $T$, preserving the condition that $d_{Lip}(T,T\Phi)=\inf\widetilde{\Phi}$, the map $f\colon T\to T\Phi$ is $G$-equivariantly homotopic to an optimal map $f'\colon T\to T\Phi$ with $\Delta(f')\subseteq\Delta$ such that $$f'\left(\Delta(f')\right)\subseteq\Delta(f').$$
\end{prop}

\begin{proof}
Suppose that $f(\Delta)$ is not contained in $\Delta$ and let $e\in E(\Delta)$ be an edge such that $f(e)\nsubseteq\Delta$. Slightly scale up the metric on $\Delta$ and down on $T\setminus\Delta$ while maintaining covolume 1. This lowers the stretching factor on $e$ and produces a new tension forest, of the original map $f$ made linear on edges, that is properly contained in the old one. Since $d_{Lip}(T,T\Phi)$ is minimal among all translation distances of $\Phi$, we will not have removed all edges of $\Delta$ and started over with a new tension forest that corresponds to a strictly smaller maximal stretching factor. In particular, there always exists an edge $e'\in E(\Delta)$ such that $f(e')\subseteq\Delta$. The stretching factor of $f$ on $e'$ remains unchanged and we preserve the condition that $d_{Lip}(T,T\Phi)=\inf\widetilde{\Phi}$. As $T$ has only finitely many $G$-orbits of edges, after finitely many repetitions we have $f(\Delta)\subseteq\Delta$. If at this point $\Delta$ has a vertex with only one gate, we perturb $f$ to an optimal map $f'$ as in the proof of Proposition~\ref{prop: existence of optimal maps}.
\end{proof}

\subsection{Parabolic automorphisms}
\label{sec: parabolics}

Let $\mathcal{PD}$ be a projectivized deformation space of $G$-trees and $T\in\mathcal{PD}$. We say that a $G$-invariant subforest $S\subseteq T$ is \emph{essential} if it contains the hyperbolic axis of some hyperbolic group element. The notion of essential $G$-invariant subforests generalizes the notion of homotopically nontrivial subgraphs of marked metric graphs in Outer space.

\begin{de}
\label{de: irreducible automorphism}
An automorphism $\Phi\in \Out_\mathcal{D}(G)$ is \emph{reducible} if there exists a $G$-tree $T\in\mathcal{PD}$ and a $G$-equivariant map $f\colon T\to T\Phi$ that leaves an essential proper $G$-invariant subforest of $T$ invariant. If $\Phi$ is not reducible, it is \emph{irreducible}.
\end{de}

As we will see, parabolic automorphisms are often reducible (Corollary~\ref{cor: parabolics reducible}). For this, let $\Phi\in \Out_\mathcal{D}(G)$ be a parabolic automorphism (i.e., $\inf\widetilde{\Phi}$ is not realized) and $(T_k)_{k\in\mathbb{N}}$ a sequence of $G$-trees in $\mathcal{PD}$ such that $\lim_{k\to\infty} d_{Lip}(T_k,T_k\Phi)=\inf\widetilde{\Phi}$. For $\Theta>0$ we denote by $\mathcal{PD}(\Theta)$ the $\Out_\mathcal{D}(G)$-invariant subspace of $\mathcal{PD}$ consisting of all $G$-trees $T\in\mathcal{PD}$ that satisfy $l_T(g)\geq\Theta$ for all hyperbolic group elements $g\in G$. We call $\mathcal{PD}(\Theta)$ the \emph{$\Theta$-thick part} of $\mathcal{PD}$.

\begin{prop}
\label{prop: parabolics leave thick part}
If the projectivized deformation space $\mathcal{PD}$ consists of irreducible $G$-trees and $\Out_\mathcal{D}(G)$ acts on $\mathcal{PD}$ with finitely many orbits of simplices then for only finitely many $k\in\mathbb{N}$ we have $T_k \in\mathcal{PD}(\Theta)$.
\end{prop}

\begin{proof} We will argue as in the proof of \cite[Claim 72]{Me11}. Suppose that the proposition is false and that, after passing to a subsequence, we have $T_k\in\mathcal{PD}(\Theta)$ for all $k\in\mathbb{N}$. We will lead this to a contradiction.

Since $\Out_\mathcal{D}(G)$ acts on $\mathcal{PD}$ with finitely many orbit of simplices, it acts on the thick part $\mathcal{PD}(\Theta)$ cocompactly in all three topologies. In particular, the image of $(T_k)_{k\in\mathbb{N}}$ in the quotient $\mathcal{PD}(\Theta)/\Out_{\mathcal{D}}(G)$ has a weakly convergent subsequence. We can thus find a sequence of outer automorphisms $(\psi_k)_{k\in\mathbb{N}}\subset \Out_{\mathcal{D}}(G)$ such that, after passing to a subsequence, $(T_k\psi_k)_{k\in\mathbb{N}}$ weakly converges in $\mathcal{PD}(\Theta)$ to some $T\in\mathcal{PD}(\Theta)$. We have
\begin{align*}
d_{Lip}(T\psi_k^{-1},T\psi_k^{-1}\Phi) &\leq d_{Lip}(T\psi_k^{-1},T_k) + d_{Lip}(T_k,T_k\Phi) + d_{Lip}(T_k\Phi,T\psi_k^{-1}\Phi)\\
&= d_{Lip}(T,T_k\psi_k) + d_{Lip}(T_k,T_k\Phi) + d_{Lip}(T_k\psi_k,T)
\end{align*}
where $\lim_{k\to\infty} d_{Lip}(T,T_k\psi_k)=\lim_{k\to\infty} d_{Lip}(T_k\psi_k,T)=0$ by Proposition~\ref{prop: convergent sequence}. Hence, $\lim_{k\to\infty} d_{Lip}(T,T\psi_k^{-1}\Phi\psi_k)=\lim_{k\to\infty} d_{Lip}(T_k,T_k\Phi)=\inf \widetilde{\Phi}.$

By Theorem~\ref{thm: inf = sup}, for all $k\in\mathbb{N}$ there exists a candidate $\xi_k\in \mathrm{cand}(T)$ such that
\begin{equation*}
\sigma(T,T\psi_k^{-1}\Phi\psi_k)=\frac{l_{T\psi_k^{-1}\Phi\psi_k}(\xi_k)}{l_T(\xi_k)}=\frac{l_T(\psi_k^{-1}\Phi\psi_k(\xi_k))}{l_T(\xi_k)}.
\end{equation*}
The translation length function of $T$ has discrete image in $\mathbb{R}$ and hence the numerator takes discrete values. Since the candidates of $T$ have only finitely many different translation lengths, the denominator takes only finitely many values and we conclude that the sequence $\left(\sigma(T,T\psi_k^{-1}\Phi\psi_k)\right)_{k\in\mathbb{N}}$ is discrete. For large $k$ we thus have $$d_{Lip}(T\psi_k^{-1},T\psi_k^{-1}\Phi)=d_{Lip}(T,T\psi_k^{-1}\Phi\psi_k)=\inf \widetilde{\Phi}$$ contradicting the assumption that $\Phi$ is parabolic.
\end{proof}

\begin{cor}
\label{cor: parabolics reducible}
Under the assumptions of Proposition~\ref{prop: parabolics leave thick part}, for large $k$ any optimal map $f\colon T_k\to T_k\Phi$ leaves an essential proper $G$-invariant subforest of $T_k$ invariant up to $G$-equivariant homotopy. In particular, every parabolic automorphism $\Phi\in \Out_\mathcal{D}(G)$ is reducible.
\end{cor}

If $T$ is a minimal $G$-tree then a subforest $S\subseteq T$ with no trivial components is a \emph{core subforest} if it does not have any vertices of valence $1$. Every $G$-invariant subforest $S\subseteq T$ with no trivial components contains a unique (possibly empty) maximal $G$-invariant core subforest $\mathrm{core}(S)\subseteq S\subseteq T$, obtained by inductively removing $G$-orbits of edges whose terminal or initial vertex has valence 1. The process of removing $G$-orbits of edges terminates after finitely many steps by the cocompactness of $T$.

\begin{proof}
For $T\in\mathcal{PD}$ and $\varepsilon>0$, let $T^\varepsilon\subseteq T$ be the union of all subsets of the form $\bigcup_{k\in\mathbb{Z}}g^{k}[x,gx]$ with $g\in G$ hyperbolic and $x\in T$ such that $d(x,gx)\leq\varepsilon$. In particular, $T^\varepsilon$ contains the axes of all hyperbolic group elements $g\in G$ with $l_T(g)\leq\varepsilon$. Although $T^\varepsilon\subseteq T$ is generally not a simplicial subcomplex of $T$, we will still speak of $T^\varepsilon$ as a (nonsimplicial) subforest, as it becomes a subcomplex after subdividing the simplicial structure on $T$. In fact, $T^\varepsilon$ has no trivial components and its maximal $G$-invariant core subforest $\mathrm{core}(T^\varepsilon)\subseteq T^\varepsilon$ will be a genuine simplicial subforest of $T$. Since $G$ acts on $T$ by isometries, if $\bigcup_{k\in\mathbb{Z}}g^{k}[x,gx]$ is contained in $T^\varepsilon$ then for all $h\in G$ the translate $$h(\bigcup_{k\in\mathbb{Z}}g^{k}[x,gx])=\bigcup_{k\in\mathbb{Z}}(hgh^{-1})^{k}[hx,hgx]$$ is contained in $T^\varepsilon$ as well. Thus, $T^\varepsilon\subseteq T$ is $G$-invariant.

Since $\Out_\mathcal{D}(G)$ acts on $\mathcal{PD}$ with finitely many orbits of simplices, the complex $\mathcal{PD}$ must be finite-dimensional, say of dimension $d\in\mathbb{N}$, and the number of $G$-orbits of edges of any $T\in\mathcal{PD}$ is bounded above by $d+1$. Because the $G$-trees in $\mathcal{PD}$ have covolume 1, in any $G$-tree $T\in\mathcal{PD}$ there exists an orbit of edges with associated edge length $\geq\frac{1}{d+1}$. Therefore, for $\varepsilon<\frac{1}{d+1}$ the subforest $T^\varepsilon\subseteq T$ is a proper subforest. Given $G$-invariant simplicial subforests $S'\subseteq S$ of $T$ with no trivial components, the subforest $S'$ is a proper subforest of $S$ if and only if $G\backslash S-G\backslash S'$ consists of at least one edge. Hence, as the $G$-trees in $\mathcal{PD}$ have at most $d+1$ $G$-orbits of edges, the number $d+1$ is a uniform bound for the length of any chain of proper $G$-invariant simplicial subforests with no trivial components of any $G$-tree in $\mathcal{PD}$.

Let $D=\inf\widetilde{\Phi}$. Moreover, let $\varepsilon<\frac{1}{d+1}$ and $\Theta=\frac{\varepsilon}{e^{(D+1)(d+1)}}$. By Proposition~\ref{prop: parabolics leave thick part}, we can choose $k$ so large that $T_k\notin\mathcal{PD}(\Theta)$ and $d_{Lip}(T_k,T_k\Phi)<D+1$. For $i=0,\ldots,d+1$, define $\delta_i=\frac{\varepsilon}{e^{(D+1)i}}$ and consider the chain of $G$-invariant subforests $$T_k^\varepsilon=T_k^{\delta_0}\supseteq T_k^{\delta_1}\supseteq\ldots\supseteq T_k^{\delta_{d+1}}=T_k^\Theta$$ all of which are proper subforests of $T_k$. Note that $T_k^\Theta\neq\emptyset$, since $T_k\notin\mathcal{PD}(\Theta)$ and thus there exists a hyperbolic group element $g\in G$ with $l_{T_k}(g)<\Theta$ whose axis lies in $T_k^\Theta$. The associated chain of core subforests is a chain of $G$-invariant simplicial subforests of $T_k$, whose number of proper inclusions is bounded by $d$ by the arguments given above. Thus, there exists $i\in\left\{0,\ldots,d\right\}$ for which we have $\mathrm{core}(T_k^{\delta_{i+1}})=\mathrm{core}(T_k^{\delta_i})$. Since $d_{Lip}(T_k, T_k\Phi)<D+1$, the Lipschitz constant of the optimal map $f\colon T_k\to  T_k\Phi$ is smaller than $e^{D+1}$ and we have $$f(\mathrm{core}(T_k^{\delta_{i+1}}))\subseteq f(T_k^{\delta_{i+1}})\subseteq T_k^{\delta_{i}}.$$ The subforest $\mathrm{core}(T_k^{\delta_{i}})\subseteq T_k^{\delta_{i}}$ is a $G$-equivariant deformation retract of $T_k^{\delta_{i}}\subseteq T_k$ and the obvious deformation retraction extends to a $G$-equivariant self homotopy equivalence $h$ of $T_k$ (this is easily seen on the level of quotient graphs of groups). Now $f$ is $G$-equivariantly homotopic to the $G$-equivariant map $T_k\stackrel{f}{\to} T_k\Phi\stackrel{h}{\to} T_k\Phi$ that leaves the proper $G$-invariant simplicial subforest $\mathrm{core}(T_k^{\delta_{i+1}})=\mathrm{core}(T_k^{\delta_i})\subset T_k$ invariant. As remarked above, there exists a hyperbolic group element $g\in G$ whose axis lies in $T_k^\Theta$ and therefore also in $\mathrm{core}(T_k^{\delta_{i}})$, and we conclude that $\mathrm{core}(T_k^{\delta_{i}})$ is essential.
\end{proof}

\subsection{Train track representatives}
\label{sec: train track representatives}

Let $\mathcal{PD}$ be a projectivized deformation space of $G$-trees and $T\in\mathcal{PD}$. Moreover, let $\Phi\in \Out_\mathcal{D}(G)$.

\begin{de}
An optimal map $f\colon T\to T\Phi$ is a \emph{train track map} if it satisfies the following three conditions:
\begin{itemize}
\item[$(i)$] $\Delta(f)=T$;
\item[$(ii)$] $f$ maps edges to legal paths (see Definition~\ref{de: train tracks});
\item[$(iii)$] If $f$ maps a vertex $v\in V(T)$ to a vertex $f(v)\in V(T\Phi)$ then it maps legal turns at $v$ to legal turns at $f(v)$. (If $v$ has 2 gates then $f(v)$ could alternatively lie in the interior of an edge. Since $f$ is linear on edges, it then maps inequivalent directions at $v$ to inequivalent directions at $f(v)$.)
\end{itemize}
\end{de}

If $f$ is a train track map then for any legal line $L\subset T$ and every $k\in\mathbb{N}$ the image $f^k(L)\subset T\Phi^k$ is again a legal line. We say that an automorphism $\Phi\in \Out_\mathcal{D}(G)$ is \emph{represented by a train track map} if there exists a $G$-tree $T\in\mathcal{PD}$ and an optimal map $f\colon T\to T\Phi$ that is a train track map.

\begin{prop}
Let $\mathcal{PD}$ be a projectivized deformation space of $G$-trees. If an automorphism $\Phi\in \Out_\mathcal{D}(G)$ is represented by a train track map $f\colon T\to T\Phi$ then $d_{Lip}(T,T\Phi)=\inf\widetilde{\Phi}$ and, in particular, $\Phi$ is nonparabolic.
\end{prop}

\begin{proof}
Our argument is a generalization of \cite[Remark~8]{Be11}. Suppose that $f\colon T\to T\Phi$ is an optimal map that is a train track map. By Theorem~\ref{thm: inf = sup} and Lemma~\ref{lem: characterization of witnesses}, there exists a hyperbolic group element $\xi\in G$ whose axis $A_\xi\subset T$ lies in $\Delta(f)$ and is legal with respect to the train track structure defined by $f$ (once we know that there exists an optimal map $f\colon T\to T\Phi$, Theorem~\ref{thm: inf = sup} no longer requires $T$ to be irreducible). Since $f$ is a train track map, for all $k\in\mathbb{N}$ the image $f^k(A_\xi)\subset T\Phi^k$ is a $\xi$-invariant line -- and thus equals the hyperbolic axis of $\xi$ in $T\Phi^k$ -- that lies in $\Delta(f)$ and is legal. We therefore have
\begin{align*}
\sigma(T,T\Phi^k)=\sup_g\frac{l_{T\Phi^k}(g)}{l_T(g)}\geq \frac{l_{T\Phi^k}(\xi)}{l_T(\xi)} &= \frac{l_{T\Phi}(\xi)}{l_T(\xi)}\cdots\frac{l_{T\Phi^k}(\xi)}{l_{T\Phi^{k-1}}(\xi)}\\
&= \sigma(T,T\Phi)\cdots\sigma(T\Phi^{k-1},T\Phi^k)=\sigma(T,T\Phi)^k
\end{align*}
from which we conclude that $\sigma(T,T\Phi^k)=\sigma(T,T\Phi)^k$. In order to show that $d_{Lip}(T,T\Phi)=\inf\widetilde{\Phi}$, let $T'\in\mathcal{PD}$ be any other $G$-tree. We have
\begin{align*}
k\cdot d_{Lip}(T,T\Phi) &= d_{Lip}(T,T\Phi^k)\\
&\leq d_{Lip}(T,T')+d_{Lip}(T',T'\Phi^k)+d_{Lip}(T'\Phi^k,T\Phi^k)\\
&\leq d_{Lip}^{sym}(T,T')+k\cdot d_{Lip}(T',T'\Phi)
\end{align*}
and hence $d_{Lip}(T,T\Phi)\leq \frac{1}{k}\cdot d_{Lip}^{sym}(T,T')+d_{Lip}(T',T'\Phi)$. Letting $k$ go to infinity, we see that $d_{Lip}(T,T\Phi)\leq d_{Lip}(T',T'\Phi)$.
\end{proof}

As for existence of train track representatives, we have the following:

\begin{thm}[Existence of train track representatives]
\label{thm: train tracks}
Let $\mathcal{PD}$ be a projectivized deformation space of irreducible $G$-trees. If $\Out_\mathcal{D}(G)$ acts on $\mathcal{PD}$ with finitely many orbits of simplices then every irreducible automorphism (see Definition~\ref{de: irreducible automorphism}) $\Phi\in \Out_\mathcal{D}(G)$ is represented by a train track map.
\end{thm}

\begin{proof}
Since the automorphism $\Phi$ is irreducible, by Corollary~\ref{cor: parabolics reducible} it is nonparabolic, i.e., $\inf\widetilde{\Phi}$ is realized. Let $T\in\mathcal{PD}$ such that $d_{Lip}(T,T\Phi)=\inf\widetilde{\Phi}$ and let $f\colon T\to T\Phi$ be an optimal map, which exists by the irreducibility of $\mathcal{PD}$. We claim that $f$ already satisfies $(i)$ and $(ii)$ of Definition~\ref{de: train tracks}:

Assertion $(i)$ immediately follows from Proposition~\ref{prop: nonparabolic tension forest invariant}, as we could otherwise slightly perturb the metric on $T$ and find an optimal map $T\to T\Phi$ that leaves an essential proper $G$-invariant subforest of $T$ invariant (the tension forest of an optimal map is always essential by Theorem~\ref{thm: inf = sup} and Lemma~\ref{lem: characterization of witnesses}), contradicting the assumption that $\Phi$ is irreducible.

As for $(ii)$, suppose that an edge $e\in E(T)$ is mapped over an illegal turn. Slightly fold the illegal turn $G$-equivariantly and scale the metric on $T$ back to covolume 1. The optimal map $f\colon T\to T\Phi$ naturally induces a $G$-equivariant map that we make linear on edges relative to the vertices of $T$. The performed perturbation lowers the stretching factor of $f$ on the edge induced by $e$, which therefore drops out of the tension forest. Each witness $A_\xi\subset\Delta(f)=T$ is legal with respect to $f$ and does not get folded, whence the stretching factor of $f$ on $A_\xi$ does not increase\footnote{In fact, there also exists a witness $A_\xi\subset T$ whose $f$-image $A_{\Phi(\xi)}\subset T$ does not get folded either, as the induced map would otherwise have strictly smaller Lipschitz constant, contradicting the fact that $d_{Lip}(T,T\Phi)$ is minimal among all translation distances of $\Phi$.}. Hence, we preserve the condition that $d_{Lip}(T,T\Phi)=\inf\widetilde{\Phi}$ and the Lipschitz constant of $f$ remains minimal among all $G$-equivariant Lipschitz maps from $T$ to $T\Phi$.  After perturbing $f$ to an optimal map as in the proof of Proposition~\ref{prop: existence of optimal maps}, we obtain an optimal map $T\to T\Phi$ whose tension forest is a proper subforest of $T$. By Proposition~\ref{prop: nonparabolic tension forest invariant}, this again contradicts the assumption that $\Phi$ is irreducible.

Finally, we may perturb $T$ and $f$ by an arbitrarily small amount, preserving the condition that $d_{Lip}(T,T\Phi)=\inf\widetilde{\Phi}$ and that $f\colon T\to T\Phi$ is an optimal map -- and therefore also preserving conditions $(i)$ and $(ii)$ -- such that $(iii)$ of Definition~\ref{de: train tracks} is satisfied as well: If $f$ maps a legal turn at a vertex $v\in V(T)$ to an illegal turn, slightly fold the illegal turn $G$-equivariantly (see Figure~\ref{fig: folding illegal turns}).
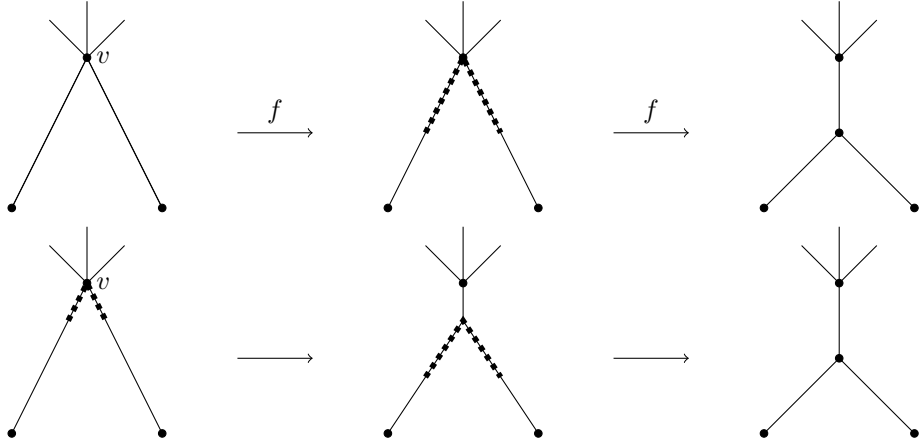
\begin{figure}[h!]
$$\begin{tikzpicture}
\draw [->] (3,-1) -- (4,-1);
\node [above] at (3.5,-1) {$f$};
\draw [->] (8,-1) -- (9,-1);
\node [above] at (8.5,-1) {$f$};

\draw [-] (1,0) -- (.5,.5);\draw [-] (1,0) -- (1,.75);\draw [-] (1,0) -- (1.5,.5);
\draw [-] (6,0) -- (5.5,.5);\draw [-] (6,0) -- (6,.75);\draw [-] (6,0) -- (6.5,.5);
\draw [-] (11,0) -- (10.5,.5);\draw [-] (11,0) -- (11,.75);\draw [-] (11,0) -- (11.5,.5);

\draw [-] (0,-2) -- (1,0) -- (2,-2);
\draw [-] (0,-2) -- (1,0) -- (2,-2);
\draw [-] [dotted, line width=0.075cm] (5.5,-1) -- (6,0) -- (6.5,-1);
\draw [-] (5,-2) -- (6,0) -- (7,-2);
\draw [-] (11,0) -- (11,-1);
\draw [-] (10,-2) -- (11,-1) -- (12,-2);
\draw [fill] (1,0) circle [radius=.05];
\draw [fill] (0,-2) circle [radius=.05];
\draw [fill] (2,-2) circle [radius=.05];
\draw [fill] (6,0) circle [radius=.05];
\draw [fill] (5,-2) circle [radius=.05];
\draw [fill] (7,-2) circle [radius=.05];
\draw [fill] (11,0) circle [radius=.05];
\draw [fill] (11,-1) circle [radius=.05];
\draw [fill] (10,-2) circle [radius=.05];
\draw [fill] (12,-2) circle [radius=.05];

\draw [-] (1,-3) -- (.5,-2.5);\draw [-] (1,-3) -- (1,-2.25);\draw [-] (1,-3) -- (1.5,-2.5);
\draw [-] (6,-3) -- (5.5,-2.5);\draw [-] (6,-3) -- (6,-2.25);\draw [-] (6,-3) -- (6.5,-2.5);
\draw [-] (11,-3) -- (10.5,-2.5);\draw [-] (11,-3) -- (11,-2.25);\draw [-] (11,-3) -- (11.5,-2.5);

\node [right] at (1,0) {$v$};
\node [right] at (1,-3) {$v$};

\draw [->] (3,-4) -- (4,-4);
\draw [->] (8,-4) -- (9,-4);
\draw [-] [dotted, line width=0.075cm] (.75,-3.5) -- (1,-3) -- (1.25,-3.5);
\draw [-] (0,-5) -- (1,-3) -- (2,-5);
\draw [-] [dotted, line width=0.075cm] (5.5,-4.25) -- (6,-3.5) -- (6.5,-4.25);
\draw [-] (6,-3) -- (6,-3.5);
\draw [-] (5,-5) -- (6,-3.5) -- (7,-5);
\draw [-] (11,-3) -- (11,-4);
\draw [-] (10,-5) -- (11,-4) -- (12,-5);
\draw [fill] (1,-3) circle [radius=.05];
\draw [fill] (0,-5) circle [radius=.05];
\draw [fill] (2,-5) circle [radius=.05];
\draw [fill] (6,-3) circle [radius=.05];
\draw [fill] (5,-5) circle [radius=.05];
\draw [fill] (7,-5) circle [radius=.05];
\draw [fill] (11,-3) circle [radius=.05];
\draw [fill] (11,-4) circle [radius=.05];
\draw [fill] (10,-5) circle [radius=.05];
\draw [fill] (12,-5) circle [radius=.05];
\end{tikzpicture}$$
\caption{Legal and illegal (dashed) turns. The upper row shows turns before folding, the bottom row after folding. The number $G(T)$ in the upper row is $3+2+(2+1)=8$, whereas in the bottom row it has decreased to $2+(2+0)+(2+1)=7$.}
\label{fig: folding illegal turns}
\end{figure}
Again, each witness $A_\xi\subset T$ is legal with respect to $f$ and does not get folded so that the stretching factor of $f$ on $A_\xi$ does not increase. Thus, we preserve the property that $d_{Lip}(T,T\Phi)=\inf\widetilde{\Phi}$ and that $f$ is a minimal stretch map. The perturbation makes the legal turn at $v$ illegal, but the induced map $f$ made linear on edges is still optimal and $v$ has still at least two gates, for $f$ would otherwise give rise to an optimal map whose tension forest is a proper subforest of $T$. The folding decreases the number $G(T)=\sum_w{\max}\left\{0,G(w)-2\right\}$, where $w$ ranges over the finitely many $G$-orbits of vertices of $T$ and $G(w)$ denotes the number of gates at $w$. After finitely many steps, we obtain an optimal map $f\colon T\to T\Phi$ that also satisfies condition $(iii)$.
\end{proof}

\begin{ex}
\label{ex: train tracks for virtually free}
Let $G$ be a finitely generated virtually nonabelian free group and $\mathcal{PD}$ the projectivized deformation space of minimal $G$-trees with finite vertex stabilizers (Example~\ref{ex: virtually free group}); it is irreducible and $\Out_{\mathcal{D}}(G)=\Out(G)$ acts on $\mathcal{PD}$ with finitely many orbits of simplices (Example~\ref{ex: virtually free}). We conclude that every irreducible automorphism of $G$ is represented by a train track map. This generalizes \cite[Theorem~1.7]{BH92} to virtually free groups.
\end{ex}

\begin{ex}
\label{ex: train tracks for GBS}
Let $G$ be a nonelementary GBS group that contains no solvable Baumslag-Solitar group $\BS(1,n)$ with $n\geq 2$. The projectivized deformation space $\mathcal{PD}$ of minimal $G$-trees with infinite cyclic vertex and edge stabilizers is irreducible (Example~\ref{ex: nonelementary GBS groups}) and $\Out_\mathcal{D}(G)=\Out(G)$ acts on $\mathcal{PD}$ with finitely many orbits of simplices (Example~\ref{ex: GBS orbits}). Hence, every irreducible automorphism of $G$ is represented by a train track map.
\end{ex}

\end{document}